\documentclass[a4paper,12pt]{amsart}
\usepackage{amssymb,amsthm,amsmath,color,url}
\usepackage[margin=3cm]{geometry}

\usepackage{graphicx}
\usepackage{stmaryrd}
\usepackage[T1]{fontenc}
\usepackage[english]{babel}
\usepackage{comment}
\usepackage{hyperref}
\usepackage{ifthen}
\hypersetup{
  colorlinks   = true,
  linkcolor = red!70!black,
     citecolor    = green!50!black
}
\usepackage[noabbrev,capitalise]{cleveref}
\crefname{subsection}{Subsection}{Subsections}
\crefname{claim}{Claim}{Claims}
\crefname{problem}{Problem}{Problems}

\makeatletter
\def\namedlabel#1#2{\begingroup
   \def\@currentlabel{#2}%
   \label{#1}\endgroup
}
\makeatother

\usepackage{thmtools}
\usepackage{thm-restate}
\usepackage{enumitem}
\usepackage{caption,subcaption}

\usepackage[myheadings]{fullpage}
\usepackage{tikz}
\usetikzlibrary{shadows}
\usetikzlibrary{backgrounds}
\usetikzlibrary{arrows}
\usetikzlibrary{patterns}
\usetikzlibrary{calc}
\usetikzlibrary{shapes}
\usetikzlibrary{positioning}
\usetikzlibrary{math}
\usetikzlibrary{external}
\usetikzlibrary{decorations.pathreplacing}
\declaretheorem[name=Theorem, numberwithin=section]{theorem}
\declaretheorem[name=Lemma, sibling=theorem]{lemma}

\declaretheorem[name=Corollary, sibling=theorem]{corollary}

\declaretheorem[name=Problem, sibling=theorem]{problem}

\declaretheorem[name=Claim, sibling=theorem]{claim}
\declaretheorem[name=Claim, numbered=no]{claim*}

\newenvironment{proofofclaim}{\noindent{\emph{Proof of the Claim}:}}{\hfill$\Diamond$\medskip}
\declaretheorem[name=Remark, style=remark, sibling=theorem]{remark}

\def\cqedsymbol{\ifmmode$\lrcorner$\else{\unskip\nobreak\hfil
\penalty50\hskip1em\null\nobreak\hfil$\lrcorner$
\parfillskip=0pt\finalhyphendemerits=0\endgraf}\fi}

\newcommand{\Cay}{\mathrm{Cay}}

\newcommand{\Aut}{\mathrm{Aut}}

\newcommand{\Stab}{\mathrm{Stab}}

\newcommand{\Sep}{(Y,S,Z)}
\newcommand{\tc}{\widetilde{c}}

\newcommand*\Sepi[1]{(Y_{#1},S_{#1},Z_{#1})}

\newcommand*\sg[1]{\{ #1 \}}

\def\C{\mathcal{C}} % class
\def\X{\mathcal{X}} % family of X
\let\le\leqslant
\let\ge\geqslant
\let\leq\leqslant
\let\geq\geqslant

\begin{document}

\title[Periodic colorings and orientations in infinite graphs]{Periodic colorings and orientations \\in infinite graphs}

\author[T.~Abrishami]{Tara Abrishami}
\address[T.~Abrishami]{Department of Mathematics, University of
  Hamburg, Hamburg, Germany}
\email{tara.abrishami@uni-hamburg.de}

\author[L.~Esperet]{Louis Esperet}
\address[L.~Esperet]{Univ.\ Grenoble Alpes, CNRS, Laboratoire G-SCOP,
  Grenoble, France}
\email{louis.esperet@grenoble-inp.fr}

\author[U.~Giocanti]{Ugo Giocanti}
\address[U.~Giocanti]{Theoretical Computer Science Department, Faculty of Mathematics and Computer Science, Jagiellonian University, Krak\'ow,
 Poland}
\email{ugo.giocanti@uj.edu.pl}

\author[M.~Hamann]{Matthias Hamann}
\address[M.~Hamann]{Department of Mathematics, University of
  Hamburg, Hamburg, Germany}
\email{matthias.hamann@math.uni-hamburg.de}

\author[P.~Knappe]{Paul Knappe}
\address[P.~Knappe]{Department of Mathematics, University of
  Hamburg, Hamburg, Germany}
  \email{paul.knappe@uni-hamburg.de}

\author[R.G.~M\"oller]{R\"ognvaldur G. M\"oller}
\address[R.G.~M\"oller]{Science Institute, University of Iceland, Reykjav\'ik, Iceland}
\email{roggi@hi.is}

%\author[Norin]{Sergey Norin}
%\address[S.~Norin]{Department of Mathematics and Statistics, McGill
%  University, Montreal, Canada}
%\email{sergey.norin@mcgill.ca}

%\author[Przytycki]{Piotr Przytycki}
%\address[P.~Przytycki]{Department of Mathematics and Statistics, McGill
%  University, Montreal, Canada}
%\email{piotr.przytycki@mcgill.ca}

\thanks{T.A. is supported by the National Science Foundation Award Number DMS-2303251 and the Alexander von Humboldt Foundation. L.E. and U.G. are partially supported by the French ANR Project
  GrR (ANR-18-CE40-0032), TWIN-WIDTH
  (ANR-21-CE48-0014-01), and by LabEx
  PERSYVAL-lab (ANR-11-LABX-0025). U. G. is supported by the National Science Center of Poland
under grant 2022/47/B/ST6/02837 within the OPUS 24 program.
  M.H. is funded by DFG, Project No. 549406527.
  P.K. is supported by a doctoral scholarship of the Studienstiftung des deutschen Volkes.}

\date{}

\begin{abstract}
  We study the existence of periodic colorings and orientations
  in locally finite graphs. A coloring or orientation of a
  graph $G$ is \emph{periodic} if the resulting colored or oriented graph is
  \emph{quasi-transitive}, meaning that $V(G)$ has finitely many
  orbits under the action of the group of automorphisms of
  $G$ preserving the coloring or the orientation. When such a periodic
  coloring
  or orientation of $G$ exists, $G$ itself must be quasi-transitive and it is
  natural to investigate when quasi-transitive graphs have such
  periodic colorings or orientations. We provide examples of Cayley
  graphs with no periodic orientation or non-trivial coloring, and examples of
  quasi-transitive graphs of treewidth 2 without periodic
  orientation or proper coloring. On the other hand we show that every
  quasi-transitive graph $G$ of bounded pathwidth has a periodic
  proper coloring with $\chi(G)$ colors and a periodic orientation. We relate
  these problems with techniques and questions from symbolic dynamics and distributed computing and conclude
  with a number of open problems.
\end{abstract}

\maketitle

\section{Introduction}

The purpose of this paper is to investigate when highly-symmetric graphs have highly-symmetric colorings or orientations. Questions of the same flavor have been raised in a number of different settings. In the emerging field of descriptive graph theory (see \cite{Pik21} for a recent survey), the graphs under study are equipped with a topology or a measure, and the goal is to find combinatorial objects, such as proper colorings, that behave well with respect to the underlying topology or measure. This defines notions such as the Borel chromatic number of a Borel graph, and these notions are then compared with the classical chromatic number for various graph classes. Another example is that of the recursive chromatic number of a recursive graph: here the graphs under consideration are recursive sets (adjacency in the graph can be computed by a Turing machine), and the goal is to find a proper coloring which is also recursive \cite{Bea76, Kie81, Sch80}.

The early motivation for studying recursive colorings of infinite (recursive) graphs is that a number of classical coloring results for infinite graphs \cite{dBE51,Got51} use compactness arguments and thus do not provide an explicit description of the resulting colorings. If a graph $G$ is highly-symmetric, for instance if $G$ is the Cayley graph of some group $\Gamma$, a simple way to give a finite description of a proper coloring of $G$ is to assign colors to a finite subset $X$ of vertices of $G$, and then transfer the colors of $X$ to all the remaining vertices using the action of a specific subgroup $\Gamma'$ of $\Gamma$ on the graph $G$. This provides an explicit, finite description of a coloring of $G$, provided for instance that $\Gamma$ and $\Gamma'$ are automatic. As a simple example of this procedure, consider any Cayley graph $G$ of $\mathbb{Z}^2$, and observe that any precoloring of any large enough square region around the identity, with pairwise distinct colors, can be extended to a proper coloring of $G$ by translation.

We now explain what we mean by a highly-symmetric graph or coloring, starting with graphs. A natural choice is the class of \emph{vertex-transitive} graphs, which are graphs $G$ in which for any two vertices $u,v$ of $G$, there is an automorphism of $G$ that maps $u$ to $v$ (in other words, the automorphism group $\Aut(G)$ of $G$ acts transitively on $G$). Note that this class contains all Cayley graphs. We will indeed consider the slightly more general class of \emph{quasi-transitive} graphs, which are graphs $G$ such that the vertex set of $G$ has finitely many orbits under the action of $\Aut(G)$. 
This means that $V(G)$ can be partitioned into $V_1,\ldots,V_k$ so that for any $1\le i \le k$ and any $u,v\in V_i$, there is an automorphism of $G$ that maps $u$ to $v$. Every vertex-transitive graph is
quasi-transitive. In this paper, all the graphs we consider will be \emph{locally
finite}, meaning that every vertex has finite degree. Note that for quasi-transitive graphs, being locally finite is equivalent to having bounded degree.

One motivation for considering quasi-transitive graphs instead of transitive graphs is that these graphs are much less regular (which allows inductive approaches involving 
tree-decompositions into simpler quasi-transitive graphs
\cite{EGL23,HamannPlanar,HamannAccessibility}) and even though the class is more general, several classical group decomposition results such as Stallings theorem still
apply to them \cite{HamannStallings22}.
The definition of quasi-transitive graph can be naturally extended to vertex-colored
(resp.\ edge-colored or oriented) graphs by restricting the group of
automorphisms to automorphisms preserving the colors of the vertices
(resp.\ the colors or orientations of the edges).

\medskip

We say that a vertex- or edge-coloring or an orientation of a graph
$G$ is \emph{periodic} if the resulting vertex- or edge-colored or
oriented graph is quasi-transitive. In particular, observe that every periodic coloring of a graph $G$ always involves a finite number of colors. Note that this definition corresponds precisely to the procedure described above: in the case of a periodic vertex coloring we only have to fix the colors of a finite number of vertices (one in each vertex-orbit for the action of the subgroup of color-preserving automorphisms), and then this coloring naturally extends to the whole graph by the action of this subgroup. Similarly, in a periodic edge-coloring or orientation, we only have to fix the colors or orientations of a finite set of edges and these extend to the whole graph by a group action.

\medskip

Note that if a periodic coloring or
orientation exists, then $G$ itself must be quasi-transitive. This motivates the following questions, which were raised
in \cite{EGL23}.

\begin{problem}[Problem 6.3 in \cite{EGL23}]\label{pro:4}
    Is it true that any locally finite quasi-transitive graph has a periodic proper vertex-coloring?
\end{problem}

\begin{problem}[Problem 6.4 in \cite{EGL23}]\label{pro:5}
    Is it true that any locally finite quasi-transitive graph has a
    periodic orientation?
  \end{problem}

The following related problem is attributed to Bowen and Lyons in \cite{Tim11}.

 \begin{problem}[Question 4.8 in \cite{Tim11}]\label{pro:pcplanar}
      Is it true that every locally finite quasi-transitive planar graph has a periodic proper vertex-coloring with 4 colors?
  \end{problem}

These questions are specific instances of a more general problem, which asks whether every quasi-transitive graph $G$
can be
decorated with a non-trivial additional structure such that the graph $G$ is still
quasi-transitive if we restrict ourselves to automorphisms that fix
the additional  structure. It should be noted that the original motivation for raising Problems \ref{pro:4} and \ref{pro:5} is a bit different from the motivation presented in this introduction (the original goal of the authors was to find periodic structures in specific quasi-transitive graphs in order to simplify some canonical decompositions).

\medskip

  Observe that a positive answer to Problem \ref{pro:4} would imply a
  positive answer to Problem \ref{pro:5}: the orientation defined from a vertex-coloring by choosing a total ordering on the colors and orienting each edge from the endpoint with the smaller color to the endpoint with the larger color is preserved by every automorphism that preserves the coloring. 
  
  %, as seen by choosing a total
  %order on the colors and then orienting each edge from the endpoint with the smaller color to the endpoint with the larger color (any color-preserving automorphism is then also orientation-preserving).

  \medskip
  
  % It can be easily observed that Problem \ref{pro:4} has a positive
  % answer when $G$ is bipartite and connected, and that Problem \ref{pro:5} has a positive
  % answer when $G$ is the Cayley graph of some finitely generated group
  % $\Gamma$ with respect to some generating set  containing no elements of
  % order 2 (i.e., elements $s\neq 1_{\Gamma}$ such that $s^{-1}=s$) (see Section \ref{sec:prel}). 

  % \medskip

  \subsection*{Main results} The goal of this note is to present examples showing that
  Problems \ref{pro:4}, \ref{pro:5} and \ref{pro:pcplanar} have negative answers.
  The first example showing that Problem \ref{pro:4} has a negative answer
  was constructed by Hamann and M\"oller. It was then observed
  by Abrishami, Esperet and Giocanti, and independently by Norin and
  Przytycki, that a variant of this example could also be used to provide a negative
  answer to Problem \ref{pro:5}. Norin and
  Przytycki furthermore showed that the examples can be chosen to be
  Cayley graphs (and thus vertex-transitive), rather than merely
  quasi-transitive. The examples are presented in Section \ref{sec:example}. They are based on the existence of
  finitely generated infinite simple groups, whose known constructions are highly
  non-trivial. As the resulting  graphs are 1-ended, a natural question is whether similar negative examples can be obtained for graphs with 2 or infinitely many ends. We indeed present a simple alternative construction of a negative answer to Problems \ref{pro:4} and \ref{pro:5} which has treewidth 2 (and is in particular planar and $\infty$-ended) and is perfect. This gives in particular a negative answer to Problem \ref{pro:pcplanar} as well, in a very strong sense (the answer remains negative for any number of colors). We also construct a graph $G$ which has a periodic proper coloring with $\chi(G)+1$ colors, but admits no periodic proper coloring with $\chi(G)$ colors.
  
  On
  the positive side, we prove that for graphs of bounded pathwidth, Problems \ref{pro:4} and \ref{pro:5} have a positive answer (and moreover a periodic proper coloring with $\chi(G)$ colors can always be obtained in this case). The setting of bounded pathwidth is
  very natural (for infinite Cayley graphs, it corresponds to the case of 2-ended groups) and has interesting connections with symbolic dynamics.

  \subsection*{Related work} Problem \ref{pro:4} has a natural dual problem:  Is it possible to find a vertex coloring so that  no non-trivial automorphism of the graph preserves the coloring?  Is it possible with just two colors?  Assuming that no non-trivial automorphism fixes all but finitely many vertices,  Babai proved that two colors suffice \cite{Bab22}.

  \medskip
  
  Probabilistic variants of
  Problems \ref{pro:4} and \ref{pro:5} have been investigated
  extensively in probability theory. In this setting, the goal is not
  to find a specific coloring which is invariant (or almost invariant)
  under automorphisms, but rather a \emph{random}  coloring which is
  invariant under automorphims (in the sense that the random process that generates the proper coloring is  invariant under automorphism). See for instance \cite{Tim24} and the
  references therein. The results there seem to be more positive than
  in our setting (in particular it is proved in \cite{Tim24} that a
  specific type of random proper coloring based on factors of iid with
  independence at distance more than 4 exists in any graph of bounded
  degree, using a bounded number of colors -- this type of random
  proper coloring is invariant under automorphisms). Note however that
  random proper colorings that are invariant under automorphisms do
  not necessarily produce proper colorings that are invariant under
  automorphisms.

  \medskip

  The connections between our problems and symbolic dynamics are
  highlighted in Section \ref{sec:pw} (we refer the reader to this
  section for the details).

  Finally, we mention an interesting connection between the topic of
  this paper and techniques from distributed computing in Section \ref{sec:prel}.

\subsection*{Organization of the paper} We start by giving some
definitions and basic results on graph theory and group theory in
Section \ref{sec:prel}. In Section \ref{sec:example}, we give our
first example providing a negative
  answer to Problems \ref{pro:4} and \ref{pro:5}. It turns out that
  the example even provides a negative answer to a much weaker
  problem, where we only seek a periodic non-trivial coloring (instead of a
  periodic proper coloring). In Section \ref{sec:tw}
  we show how to construct graphs of bounded treewidth with no
  periodic orientation, providing a negative
  answer to Problems \ref{pro:4} and \ref{pro:5} even for graphs of
  bounded treewidth. In Section
  \ref{sec:pw} we show that Problems \ref{pro:4} and \ref{pro:5} have
  a positive answer if we restrict ourselves to graphs of bounded
  pathwidth. This is done using connections between our problems
  and classical results on subshifts of finite type in symbolic
  dynamics.
  We conclude in Section \ref{sec:ccl} with a discussion and a number
  of open problems.

  \section{Preliminaries}\label{sec:prel}

\subsection*{Graphs}
A vertex-coloring of a graph $G$ is \emph{proper} if every two adjacent
vertices in $G$ are assigned distinct colors. The \emph{chromatic
  number} of $G$, denoted by $\chi(G)$, is the minimum number of
colors in a proper vertex-coloring of $G$. A \emph{path} is a connected acyclic graph in which every vertex has degree at most $2$.

\medskip

A {\it tree-decomposition} of a graph $G$ is a pair $(T,\X)$ such that $T$ is a tree and $\X$ is a collection $(X_t: t \in V(T))$ of subsets of $V(G)$, called the {\it bags}, such that
	\begin{itemize}
		\item $\bigcup_{t \in V(T)}X_t = V(G)$,
		\item for every $e \in E(G)$, there exists $t \in V(T)$ such that $X_t$ contains the endpoints of $e$, and
		\item for every $v \in V(G)$, the set $\{t \in V(T): v \in X_t\}$ induces a connected subgraph of $T$.
	\end{itemize}
        For a tree-decomposition $(T,\X)$ as above, the sets $ X_t \cap X_{t'}
        $ for $tt' \in E(T)$ are called the \emph{adhesion sets} of $(T,\X)$, and the {\it width} of $(T,\X)$ is $\sup_{t \in V(T)}\lvert X_t \rvert-1$.
The {\it treewidth} of $G$ is the minimum width of a
tree-decomposition of $G$.

If the tree $T$ in a tree-decomposition $(T,\X)$ of a graph $G$ is a
path, then $(T,\X)$ is called a \emph{path-decomposition} of $G$, and
the \emph{pathwidth} of $G$ is the minimum width of a
path-decomposition of $G$.

\medskip

A
\emph{ray} in an infinite graph $G$ is an infinite one-way path in
$G$. Two rays of $G$ are said to be 
\emph{equivalent} if there are infinitely many disjoint paths between them in
$G$.   
An \emph{end} of $G$ is an
equivalence class of rays in $G$. It is known that the number of ends
of a connected locally finite quasi-transitive graph is either $0,1,2$ or $\infty$ \cite[Proposition 2.1]{Babai97}.
%(it was originally proved for vertex-transitive graphs \cite[Corollary 4]{Diestel93}, and a simple way to extend it to quasi-transitive graphs is to observe that the number of ends of a locally finite graph is a quasi-isometric invariant, and that every quasi-transitive graph is quasi-isometric to a transitive graph \cite[Proposition 5.1]{Woess94}). 
%\cite{Diestel93}. 
A graph with $k$ ends (for $k\in \mathbb{N}\cup \{\infty\}$)
is said to be \emph{$k$-ended}.
For every finite set $S\subseteq V(G)$ and every ray $r$, note that there exists a unique connected component of $G-S$ containing an infinite number of vertices of $r$. If $X$ denotes such a component, we say that $r$ \emph{lives} in $X$. We say that an end $\omega$ \emph{lives} in a component $X$ of $G-S$ if some (and thus all) of its rays live in $X$.

\subsection*{Groups}
  The \emph{index} of a subgroup $\Gamma'$ in a group $\Gamma$ is the
  cardinality of the family of \emph{left cosets} $\{g\Gamma':g \in
  \Gamma\}$, where  $g\Gamma'=\{gh:h \in \Gamma'\}$. Equivalently, the
  index of $\Gamma'$ in $\Gamma$ is the
  cardinality of the family of \emph{right cosets} $\{\Gamma'g:g \in
  \Gamma\}$, where  $\Gamma'g=\{hg:h \in \Gamma'\}$. If a group
  $\Gamma$ has a subgroup $\Gamma'$ of finite index then we say that
  $\Gamma$ is \emph{virtually} $\Gamma'$.  Furthermore, if $\mathcal{P}$ is a group theoretic property then $\Gamma$ is said to be \emph{virtually} $\mathcal{P}$ if $\Gamma$ contains a subgroup of finite index that has property $\mathcal{P}$.

  \medskip

  Given a finitely generated group $\Gamma$ and a finite set of
generators $S$ not containing the identity element $1_\Gamma$, the \emph{(left) Cayley graph} of $\Gamma$ with respect to
the set of generators $S$ is the graph 
$\Cay(\Gamma,S)$ whose vertex set is the set of elements of $\Gamma$
and where for every two elements $g,h\in \Gamma$ there is an edge
between $g$ and $h$ if and only $hg^{-1}\in S\cup S^{-1}$. Whenever we talk about a Cayley graph $\Cay(\Gamma,S)$, we always assume that $S$ is finite and does not contain $1_\Gamma$ (even if this is not stated explicitly). For a graph class $\C$, we say that a finitely generated group
is in $\C$ if it has a Cayley graph that is in $\C$.

\medskip

  The number of ends of a Cayley graph of a finitely generated group
  does not depend on the choice of generators (it follows from the observation that the number of ends of a locally finite graph is a quasi-isometric invariant), so we can also talk
  about the number of ends of a group. A group is 0-ended if and only
  if it is finite. 
  A finitely generated group is 2-ended if and only if it is virtually
  $\mathbb{Z}$ 
  \cite{Hopf44}, if and only if it has finite pathwidth. A finitely generated group has finite treewidth if and only if it is virtually free \cite{Pichel09}. 

 %It is known that
%a group has finite treewidth if and only if it is virtually free
%(i.e., it contains a free subgroup of finite index), and a group has
%finite pathwidth if and only if it is virtually $\mathbb{Z}$ (if and
%only if it is 2-ended, as stated above).
  \medskip

   A subgroup $\Gamma'$ of a group $\Gamma$ is \emph{normal} if $g
  g'g^{-1}\in \Gamma'$ for every $g\in \Gamma$ and $g'\in \Gamma'$
  (equivalently, the left and right cosets  $g\Gamma'$ and $\Gamma'
  g$ are equal for every $g\in \Gamma$).  A group
  $\Gamma$ is \emph{simple} if the only normal subgroups of $\Gamma$
  are $\Gamma$ itself and the trivial subgroup $\{1_\Gamma\}$ consisting of the
  identity element $1_\Gamma$ of $\Gamma$.

\medskip

  We will need the following basic
  result (which is standard, but we include a proof for the sake of completeness).

  \begin{lemma}\label{lem:1}
Let $\Gamma'$ be a subgroup of finite index in a group $\Gamma$. Then
$\Gamma$ has a normal subgroup $\Gamma''$ of finite index such that
$\Gamma''$ is a subgroup of $\Gamma'$.
  \end{lemma}
  
  \begin{proof}
The action of $\Gamma$ on the right cosets $\mathcal{C}=\{\Gamma'g:g\in \Gamma\}$ of $\Gamma'$ induces a permutation of $\mathcal{C}$ (where each $h\in \Gamma$ maps each right coset $\Gamma'g$ to $\Gamma'gh$), and this action is a group homomorphism $f$ from $\Gamma$ to $\mathrm{Sym}(\mathcal{C})$, the symmetric group on $\mathcal{C}$. Moreover, since $\Gamma'$ has finite index, $\mathrm{Sym}(\mathcal{C})$ is finite. The kernel $\ker f$ of $f$ consists of all elements $g\in \Gamma$ which stabilize every right coset of $\Gamma'$ (in particular, $\ker f$ is contained in $\Gamma'$). By the first isomorphism theorem, $\ker f$ is a normal subgroup of $\Gamma$ and the quotient group $\Gamma/\ker f$ is isomorphic to the image of $f$ (which is a finite subgroup of $\mathrm{Sym}(\mathcal{C})$). Hence, $\ker f$ is a subgroup of finite index in $\Gamma$.
  \end{proof}

  This directly implies the following classical property of infinite simple groups.

   \begin{corollary}\label{cor:2}
     Let $\Gamma$ be an infinite simple group. Then $\Gamma$ has
     no proper subgroup of finite index.
  \end{corollary}
  
  \begin{proof}
    Assume for the sake of contradiction that $\Gamma$ has
     proper subgroup $\Gamma'$ of finite index. By Lemma \ref{lem:1},
    there is a normal subgroup $\Gamma''$ of finite index in $\Gamma$
    which is a subgroup of $\Gamma'$. Note that since $\Gamma'$ is a
    proper subgroup of $\Gamma$, $\Gamma''$ is also a proper subgroup
    of $\Gamma$. As $\Gamma$ is a simple group, $\Gamma''$ is the
    singleton consisting of the identity element of $\Gamma$. Since
    $\Gamma$ is infinite, we obtain that $\Gamma''$ has infinite index in $\Gamma$,
    which is a contradiction.
  \end{proof}

\subsection*{Periodicity and strong periodicity}

 We say that a vertex- or edge-coloring or an orientation of a graph $G$ is
 \emph{periodic} if the subgroup of all automorphisms of $G$ preserving the
 coloring or the orientation acts \emph{quasi-transitively} on $V(G)$,
 meaning that $V(G)$ has finitely many orbits under the action of the
 color-preserving (or orientation-preserving) automorphisms of $G$. In
 other words, there is a partition of $V(G)$ into finitely many
 subsets $V_1,\ldots,V_k$ such that for every $1\le i \le k$ and every
 $u,v\in V_i$, there is a color-preserving (or orientation-preserving)
 automorphism of $G$ that maps $u$ to $v$. Here, we say that an automorphism of a vertex-colored graph is color-preserving if the image of any vertex is a vertex of the same color. Similarly, a color-preserving automorphism of an edge-colored graph maps any edge to an edge of the same color, and an orientation-preserving automorphism of an oriented graph maps every edge to an edge with the same orientation. 

 \medskip

 We say that a vertex- or edge-coloring or an orientation of a graph $G$ is
 \emph{strongly periodic} if the subgroup of all automorphisms of $G$ preserving the
 coloring (or the orientation) has finite index in $\Aut(G)$, the group of automorphisms of $G$. Strongly
 periodic colorings naturally arise in symbolic dynamics (see Section
 \ref{sec:pw}) and are related to periodic colorings via the following
 simple results.

 \begin{lemma}\label{lem:sptop}
For every quasi-transitive graph $G$, every subgroup $\Gamma$ of finite index
of $\Aut(G)$
acts quasi-transitively on $V(G)$. 
\end{lemma}

\begin{proof}
  Let $v_1,\ldots,{v_k}\in V(G)$ be finitely many representatives of
  the orbits of $V(G)$ under the
  action of $\Aut(G)$, and let $\Gamma g_1,\ldots,\Gamma g_\ell$ be
  the finitely many right cosets of $\Gamma$, where $g_j\in
  \Aut(G)$ for any $1\le j \le \ell$. Note that for every $1\leq i\leq k$ and $1\leq j\leq \ell$, the group $\Gamma$ acts transitively on the set $\Gamma g_j(v_i):=\sg{gg_j(v_i): g\in \Gamma}$. As the family of sets $\sg{\Gamma g_j(v_i): 1\leq i\leq k, 1\leq j\leq \ell}$ covers the vertex set $V(G)$, it follows that $\Gamma$ acts quasi-transitively on $V(G)$,
  as desired.
%  $V(G)$ is partitioned by
%  the finitely many $\Gamma$-orbits of
%  $g_j(v_i)$, for $1\le i \le k$ and $1\le j \le \ell$. It follows that $\Gamma$ acts quasi-transitively on $V(G)$,
%  as desired. 
\end{proof}

The converse of Lemma \ref{lem:sptop} does not hold in general. To see this,
consider the free group $F_{a,b}$ on 2 generators $a$ and $b$, and let $T = \Cay(F_{a, b}, \{a, b\})$ be the
natural Cayley graph associated to $F_{a,b}$ and its two generators
($T$ is the infinite 4-regular tree). Now $F_{a,b}\subseteq \Aut(T)$ acts
transitively on $V(T)$, but as $\Aut(T)$ is not
virtually free, $F_{a,b}$ does not have finite index in $\Aut(T)$.

\medskip

By taking $\Gamma$ to be the subgroup of color-preserving or
orientation-preserving automorphisms of $G$, we  obtain the following immediate consequence of Lemma \ref{lem:sptop}.

\begin{corollary}\label{cor:sptop}
Every strongly periodic coloring or orientation of a quasi-transitive graph $G$ is periodic.
\end{corollary}

\subsection*{Graphical
regular representations} Consider a Cayley graph $G=\Cay(\Gamma,S)$ of a finitely
generated group $\Gamma$. Note that $\Gamma$ naturally acts by right-multiplication on $G$ (for any $g\in \Gamma$, the map $h\mapsto hg$ induces an automorphism of $G$). The Cayley graph $G=\Cay(\Gamma,S)$ is said to be a \emph{graphical
regular representation} of $\Gamma$ if $\Gamma=\Aut(G)$ (in other words, the only automorphisms of $G$ are the right-multiplications by elements of $\Gamma$).
We now prove that for graphical
regular representations of groups, the converse of Lemma \ref{lem:sptop}
holds.

 \begin{lemma}\label{lem:ptosp}
   Let $G$ be a graphical
   regular representation of a finitely generated group. Then every
   subgroup of $\Aut(G)$ acting quasi-transitively on $V(G)$
   has finite index in $\Aut(G)$. 
\end{lemma}

\begin{proof}
Let $\Gamma$ be a finitely generated infinite group and $S$ be a
    finite set of generators of $\Gamma$ such that the automorphism group of the Cayley graph
    $G=\Cay(\Gamma,S)$ is equal to $\Gamma$. Let $\Gamma'$
    be a  subgroup of $\Aut(G)=\Gamma$ acting quasi-transitively on
    $V(G)=\Gamma$. The finitely many orbits of $V(G)=\Gamma$ under the
    action of $\Gamma'$ are precisely the left cosets of
    $\Gamma'$. It follows that
    $\Gamma'$ has finite index in $\Gamma=\Aut(G)$.
\end{proof}

As before, we obtain the following immediate corollary, which shows
that the notions of periodicity and strong periodicity
are equivalent in the case of graphical
regular representations of groups.

\begin{corollary}\label{cor:ptosp}
   Let $G$ be a graphical
regular representation of a finitely generated group. Then every periodic
    coloring or orientation of $G$ is strongly periodic. 
  \end{corollary}

 Next, we discuss an observation
  on the minimum number of colors in a periodic proper vertex-coloring of a graph
  $G$ (when such a periodic proper vertex-coloring exists). This is
  connected to classical results and techniques in distributed computing.

   \subsection*{Color reduction} The maximum degree of a graph $G$ is
   denoted by $\Delta(G)$. Every locally finite
   quasi-transitive graph has bounded maximum degree. Using a
   classical technique from distributed computing, we show that every periodic proper vertex-coloring of a graph $G$
   with more than $\Delta(G)+1$ colors can be turned into a periodic proper
   vertex-coloring of $G$ with one less color, and thus after several
   iterations into a periodic proper
   vertex-coloring of $G$ with  $\Delta(G)+1$ colors.

  \begin{lemma}\label{lem:cr}
  Let $G$ be a locally finite graph with a periodic proper vertex-coloring. Then $G$ has a periodic proper vertex-coloring with at most $\Delta(G)+1$ colors.
  \end{lemma}

  \begin{proof}
As $G$ is quasi-transitive and locally finite, it has finite maximum degree $\Delta=\Delta(G)$. Consider a periodic proper vertex-coloring $c$ of $G$ with $N>\Delta+1$ colors (call them $1,\ldots,N$). Consider the set $S$ of vertices colored with color $N$ (and note that $S$ is an independent set in $G$). The set $S$ is divided into finitely many orbits $S_1,\ldots,S_k$ under the action of the group of color-preserving automorphisms of $G$. For each vertex $v\in S$, assign to $v$ the smallest color from the set $1,\ldots,N$ that does not appear in its neighborhood, and for each vertex $v\not\in S$, leave the color of $v$ unchanged. Let $c'$ be the resulting vertex-coloring of $G$. Since $N>\Delta+1$, for each vertex of $S$, the smallest color $c'(v)$ that does not appear in its neighborhood is an element of $1,\ldots,N-1$, so $c'$ uses at most $N-1$ colors. Moreover, for each $1\le i \le k$, all the vertices of $S_i$ are recolored with the same color (as all the vertices of $S_i$ see the same set of colors in their neighborhood), so every automorphism of $G$  preserving $c$ also preserves $c'$, and thus $c'$ is periodic.  Finally, since $S$ is an independent set, the vertex-coloring $c'$ is proper (each vertex has been recolored with a color distinct from that of its neighbors, and no two adjacent vertices have been recolored in the process).

This shows that given a periodic proper vertex-coloring of $G$ with $N>\Delta+1$ colors, we can produce a periodic proper vertex-coloring of $G$ with at most $N-1$ colors. Iterating this procedure at most $N-\Delta-1$ times, we obtain a periodic proper vertex-coloring of $G$ with at most $\Delta+1$ colors, as desired.
  \end{proof}

  \begin{remark}
     In distributed computing, much faster procedures reducing the number of colors than that of the proof of Lemma \ref{lem:cr} are usually needed. Together with the color reduction technique described above, a classical algorithm of Cole and Vishkin~\cite{CV86} can be used to produce a periodic proper $(\Delta+1)$-vertex-coloring of $G$ in $\Delta^{O(\Delta)}+\log^* N$ iterations starting from a periodic proper $N$-vertex-coloring of $G$, compared to the $N-\Delta-1$ iterations needed in the proof of Lemma \ref{lem:cr}. The idea is to view the colors $1,\ldots,N$ as $O(\log N)$-bit words, and to record, for each vertex $v$ and neighbor $u$ of $v$, a bit in the color of $v$ that differs from the corresponding bit in the color of $u$. Using this, we can obtain a periodic proper $O(\log N)$-vertex-coloring of $G$ in a single iteration, and a periodic proper $\Delta^{O(\Delta)}$-vertex-coloring in $\log^* N$ iterations.
 \end{remark}

 \subsection*{Basic examples of periodic colorings and orientations.}
 We conclude this section of preliminary results with two simple lemmas describing special cases where Problems \ref{pro:4} and \ref{pro:5} have positive answers. 

\begin{lemma}
    Let $G$ be a connected quasi-transitive bipartite graph (not necessarily locally finite). Then, $G$ has a periodic proper vertex-coloring with 2 colors. 
\end{lemma}
\begin{proof}
Let $\sg{X_1, X_2}$ be the bipartition of $G$ ($\sg{X_1, X_2}$ is unique since $G$ is connected). Let $c$ be the proper 2-coloring of $G$ where for $i=1,2$, vertices from $X_i$ are assigned color $i$. For each $i=1,2$, any two vertices of $X_i$ are at even distance apart. As automorphisms preserve distances (and thus the parity of distances), for every $i=1,2$, every automorphism of $G$ that maps some vertex colored $i$ to a vertex colored $i$ is color-preserving (i.e., all the vertices of $X_i$ are mapped to $X_i$ for every $i=1,2$).

Let $V_1, \ldots,V_k$ be the orbits of $V(G)$ under the action of $\Aut(G)$, and for each $1\le j \le k$ and each $i=1,2$, let $V_j^{(i)}$ be the set of vertices of $V_j$ colored $i$. The paragraph above implies that the orbits of $V(G)$ under the action of the subgroup of color-preserving automorphisms are precisely the sets  $V_j^{(i)}$, for $1\le j \le k$ and $i=1,2$. As this collection is finite, $c$ is periodic.
%
%Fix $v \in X$, and let $\Gamma$ consist of every automorphism $g$ of $G$ such that $g(v) \in X$. Let $g \in \Gamma$ and $u \in X$. Since automorphisms are distance-preserving, the distance between $u$ and $v$ has the same parity as the distance between $g(u)$ and $g(v)$. It follows that $g(u) \in X$ if and only if $u \in X$. Therefore, every automorphism of $\Gamma$ respects the partition $(X, Y)$, seen as a proper vertex-coloring with 2 colors. 
%
%Next, we observe that $\Gamma$ is a group: it contains the identity; for all $g \in \Gamma$, it holds that $g^{-1}(v) \in X$ since $v \in X$, so $g^{-1} \in \Gamma$; and for $g_1, g_2 \in \Gamma$, it holds that $g_1(g_2(v)) \in X$ since $g_2(v) \in X$, so $\Gamma$ is closed.  
%
%Finally, we show that $\Gamma$ acts quasi-transitively on $G$. Let $g_1, g_2 \not \in \Gamma$. For $i = 1, 2$, since $g_i \not \in \Gamma$, it follows that $g_i(v) \in Y$, and so, since automorphisms are distance-preserving, $g_i(u) \in Y$ for every $u \in X$ and $g_i(w) \in X$ for every $w \in Y$. But now $g_1(g_2(v)) \in X$, so $g_1g_2 \in \Gamma$. Therefore, $\Gamma$ has index 2 in $\Aut(G)$. By Lemma \ref{lem:sptop}, $\Gamma$ acts quasi-transitively on $G$. 
%
%Since $\Gamma \subseteq \Aut(G)$ preserves $(X, Y)$ and acts quasi-transitively on $G$, we conclude that $(X, Y)$ is a periodic proper vertex-coloring of $G$ with 2 colors. 
\end{proof}

We say that an automorphism $g$ of $G$ \emph{inverts an edge $uv\in E(G)$} if $g(u)=v$ and $g(v)=u$. 
It is straightforward to observe that for any subgroup $\Gamma$ of automorphisms that preserve some proper vertex-coloring or orientation of a graph $G$, $\Gamma$ does not contain any automorphism that inverts an edge of $G$. In the next lemma it is shown that this condition is also sufficient to prove the existence of a periodic orientation.
\begin{lemma}
\label{lem:edge-inversion}
    Let $G$ be a locally finite quasi-transitive graph. Then, $G$ has a periodic orientation if and only if there is a subgroup $\Gamma \subseteq \Aut(G)$ acting quasi-transitively such that no automorphism of $\Gamma$ inverts an edge of $G$. 
\end{lemma}
\begin{proof} 
    If $G$ has a periodic orientation, then no element of the automorphism group that preserves the orientation inverts an edge of $G$ (otherwise it would not preserve the orientation of that edge). Conversely, suppose $\Gamma \subseteq \Aut(G)$ acts quasi-transitively and is such that no element of $\Gamma$ inverts an edge of $G$. We will construct an orientation preserved by $\Gamma$. Let $E_1, \hdots, E_k$ be the orbits of $E(G)$ under $\Gamma$ (since $\Gamma$ acts quasi-transitively on $V(G)$ and $G$ is locally finite, $\Gamma$ also acts quasi-transitively on $E(G)$). Fix an edge $e_i \in E_i$ and an orientation $\overrightarrow{e_i}$ of $e_i$ for all $i = 1, \hdots, k$. For every $e_i' \in E_i$, since no element of $\Gamma$ inverts an edge of $G$, there is exactly one orientation $\overrightarrow{e_i'}$ of $e_i'$ so that an automorphism of $\Gamma$ maps $\overrightarrow{e_i}$ to $\overrightarrow{e_i'}$. Now, the fixed orientations of $e_1, \hdots, e_k$ define an orientation of $E(G)$ that is preserved by $\Gamma$. 
\end{proof}

A corollary of Lemma \ref{lem:edge-inversion} is that Problem \ref{pro:5} has a positive answer when $G$ is the Cayley graph of some finitely generated group $\Gamma$ with respect to a generating set containing no elements of order 2 (i.e., elements $s \neq 1_{\Gamma}$ such that $s^{-1}=s$).
  
\section{Periodic non-trivial colorings}\label{sec:example}

  A vertex-coloring of a graph $G$ is said to be \emph{trivial} if for each
  connected component $C$ of $G$, all the vertices of $C$ are assigned the
  same color. The vertex-coloring is said to be \emph{non-trivial}
  otherwise. A
  vertex-coloring of $G$ is trivial if and only if for every edge $uv$
  of $G$, $u$ and $v$ are assigned the same color. Hence, if $G$
  contains at least one edge, for every non-trivial coloring of $G$
  there is a pair $u,v$ of adjacent vertices which are assigned
  different colors. In particular, if $G$ contains at least one edge,
  then every proper vertex-coloring of $G$ is non-trivial.

  \begin{remark}
For non-trivial vertex-colorings, color reduction (as
described in the previous section) is much simpler. Given a non-trivial vertex-coloring $c$ with colors $1,\ldots,N$ we obtain a new coloring $c'$ as follows: for each connected component $C$, we consider the color of smallest index appearing in $C$ and rename it color 1. We then recolor the vertices of the other colors in the component with color 2.  The resulting coloring is non-trivial, and if $c$ is periodic then the resulting coloring is also periodic. It follows that a graph has a periodic non-trivial vertex-coloring if and only if it has a periodic non-trivial vertex-coloring with 2 colors.
  \end{remark}

  \medskip

  Recall that Problems
 \ref{pro:4} and \ref{pro:5}  ask whether every locally finite quasi-transitive graph
 has a periodic proper vertex-coloring and  a periodic orientation. We now
 give a negative answer for graphical
regular representations of  finitely generated infinite simple groups using
 the results from the previous section, relating periodicity and
 strong periodicity. For vertex-colorings, we show that even if the notion of
 proper coloring is replaced by the much weaker notion of
 non-trivial coloring, the answer to Problem
 \ref{pro:4} remains negative.

  \begin{lemma}\label{lem:2}
    Let $G=\Cay(\Gamma,S)$ be a graphical
regular representation of a finitely generated infinite simple group $\Gamma$, for some
finite set $S$ of generators.
  Then every periodic vertex-coloring of $G$ is
  trivial. Moreover, if $S$ contains an element of order 2, then $G$
  does not have a periodic orientation.
  \end{lemma}
  
  \begin{proof}
    For the sake of contradiction, consider some non-trivial vertex-coloring $c$ of $G$ which is periodic, and denote by $\Gamma'$
    the subgroup of $\Aut(G)=\Gamma$ of automorphisms of $G$ preserving the colors
    of the vertices of $G$. By Lemma~\ref{lem:ptosp}, $\Gamma'$ has finite index in $\Gamma$.

    Since $c$ is non-trivial, there exist vertices $u,v$ of $G$ such that $c(u)\ne c(v)$. As $G$ is a Cayley graph, it
    is vertex-transitive and thus there is an automorphism $g\in
    \Aut(\Gamma)$ that maps $u$ to $v$. This automorphism $g$ does not
    preserve the colors of the vertices, so this shows that $\Gamma'$
    must be a proper
    subgroup of $\Gamma$. By Corollary \ref{cor:2}, this contradicts
    the fact that $\Gamma'$ has finite index in $\Gamma$.

    \medskip

    It now remains to prove the second part of the statement. So assume
    that $S$ contains some element $s$ of order 2. Consider any $g\in
    \Gamma=V(G)$, and its neighbor $h=sg$ in $G$.  As $s$ has order 2,
    the corresponding automorphism of $G$ maps $g$ to $h$ and $h$ to
    $g$, so it maps the pair $(g,h)$ to $(h,g)$. If $G$ has a
    periodic orientation, then the subgroup of orientation-preserving automorphisms of $G$ does not contain $s$, and is thus a
    proper subgroup of $\Gamma$. Using Corollary \ref{cor:2} again, the same argument as in the previous
    paragraph then implies that no such periodic orientation
    of $G$ exists.
  \end{proof}

  \begin{remark}
    As any proper vertex-coloring of a graph containing at least one edge is non-trivial, Lemma \ref{lem:2} implies in particular that no graph satisfying the assumptions of the lemma has a periodic proper vertex-coloring,
    and thus any graph satisfying the conditions of
  Lemma \ref{lem:2} provides a negative answer to Problem \ref{pro:4}. 
  \end{remark}

  It remains to prove that a Cayley graph satisfying the conditions of
  Lemma \ref{lem:2} exists. A specific example comes from R.\ Thompson's group
  $V$, which is a finitely presented infinite simple group. We omit the precise definition of $V$ in this paper, and instead refer the interested reader to the lecture notes \cite{CFP96} for more details on R.\ Thompson's groups. It was proved by Leemann and de la Salle \cite{LS22a} that any finitely generated group which is not virtually abelian has a graphical regular representation (see also \cite{God80} for related result in the finite case). It immediately follows that $V$ has a graphical regular representation. However, for our applications we need some control on the order of the elements in the generating set. This is achieved in the following lemma, which uses earlier results of Leemann and de la Salle \cite{LS21} on graphical regular representations, together with specific presentations of R.\ Thompson's group $V$ discovered by Bleak and Quick~\cite{BQ17}.
  
  \begin{lemma}\label{lem:V}
R.\ Thompson's group $V$ has two finite sets $S$ and $T$ of generators such that $S$ contains an element of order 2, $T$ does not contain any element of order 2, and both $\Cay(V,S)$ and $\Cay(V,T)$ are graphical regular representations of $V$.
  \end{lemma}

  \begin{proof}
  Bleak and Quick~\cite[Theorem 1.2]{BQ17} gave a presentation of $V$ with three generators $a,b,c$, two of them having order 2. Using the fact that $V$ has rank 2 and contains elements of arbitrarily large order,  \cite[Corollary 11]{LS21} by Leemann and de la Salle implies that the set $\{a,b,c\}$ of  generators of $V$ is a subset of a finite set $S$ of generators of $V$ such that $\Cay(V,S)$ is a graphical regular representations of $V$.

  Bleak and Quick~\cite[Theorem 1.3]{BQ17} also gave a presentation of $V$ with two generators $u,v$, none of them having order 2. Using \cite[Theorem 9]{LS21}, this implies that there is a finite set $T$ of generators of $V$ containing $u$ and $v$, such that $T$ does not contain any element of order 2 and $\Cay(V,T)$ is a graphical regular representations of $V$, as desired.
  \end{proof}
  
  Using Lemma \ref{lem:2}, this immediately implies the following.

  \begin{corollary}
      R.\ Thompson's group $V$ has a finite generating set $S$ such that $\Cay(V,S)$ has no periodic orientation and no periodic non-trivial vertex-coloring.
  \end{corollary}

  \begin{remark} We have chosen R. Thompson's group $V$ for simplicity but it should be mentioned that other families of infinite finitely generated simple groups provide even more striking negative answers to Problem \ref{pro:4} (and Problem \ref{pro:5}, with some additional work). For instance, if we consider again the setting of non-trivial vertex-coloring, we can take a graphical regular representation $G$ of a \emph{Tarski monster group} of exponent $p$ (see \cite[Theorem 14]{LS21} for the existence of such a representation for sufficiently large $p$). This group is infinite, finitely generated, and has the property that any proper subgroup is isomorphic to the cyclic group of order $p$. As in the proof of Lemma \ref{lem:2}, observe that for any non-trivial vertex-coloring of $G$, the subgroup of $\Aut(G)$ of color-preserving automorphisms of $G$ must be a proper subgroup of $\Aut(G)$, and thus it must be finite (which is  stronger than being of infinite index).
  
      It was pointed out to us by Emmanuel Jeandel that another infinite finitely generated simple group, constructed by Osin \cite{Osi10}, provides a negative answer to Problem \ref{pro:4} which is even more remarkable. The group $\Gamma$ constructed by Osin is an infinite, finitely generated simple group  with the property that any two non-identity elements are conjugate. Consider any graphical regular representation $G=\Cay(\Gamma,S)$ of such a group, for some finite generating set $S$ (the fact that such a representation exists follows directly from \cite[Theorem 1.1]{LS22a}, since infinite simple groups are not virtually abelian), and a proper vertex-coloring $c$ of $G$. Let $g\in \Aut(G)=\Gamma$ be any non-trivial automorphism of $G$, and let $s\in S$. As $g$ and $s$ are conjugate, there exists $h\in \Gamma$ such that $gh=hs$. So the automorphism $g$ of $G$ maps  $h$ to its neighbor $hs$ in $G$. As $c(h)\ne c(hs)$, $g$ is not color-preserving. It follows that the only color-preserving automorphism of 
      $G$ is the identity. Hence, $G$ is an infinite Cayley graph with the property that for every proper vertex-coloring  of $G$, the subgroup of color-preserving automorphisms of $G$ is trivial! We note that this seems to only apply to proper vertex-colorings (and not to the more general setting of non-trivial vertex-colorings).
  \end{remark}

  %\subsection*{The Thompson group \texorpdfstring{$V$}{V}}

  %We will consider a presentation of
  %the Thompson group $V$ given by Bleak and Quick~\cite{BQ17},

  %\medskip
  
  %TODO

 % \medskip

%We can also use recent results of Leemann and de la Salle \cite{LS21,LS22a}, who have completely described the finitely generated infinite groups admitting a graphical
%regular representation as all finitely generated infinite groups which are neither abelian nor generalized dicyclic. So all such groups that are simple provide a negative answer to Problem \ref{pro:4} (for some of their Cayley graphs). Moreover, all such groups that are simple and have an element of order 2 provide a negative answer to Problem \ref{pro:5} (again, for some of their Cayley graphs).\textcolor{green}{M: The involution need to be in the generating for the GRR. Do the cited results guarantee that?}

%\begin{remark}\label{rem:1} It was proved in \cite{JM13} that there exist finitely generated infinite simple groups that are amenable.  Using the paragraph above, this shows that Problem \ref{pro:4} still has a negative answer if we restrict ourselves to Cayley graph of amenable groups. If the groups under consideration have elements of order 2, this would also provide a negative answer to Problem \ref{pro:5} for amenable groups.
%\end{remark}

  \subsection*{Alternative version (without graphical regularity)}
  In the previous paragraphs we have used specific graphical
regular representations (that is, specific Cayley graphs) of  simple groups. We now explain how to turn any Cayley graph of a finitely generated infinite simple group into a quasi-transitive graph with no periodic proper vertex-coloring (and without periodic orientation, with an additional assumption). This does not require any result on graphical regularity, but the downside is that the resulting examples are only quasi-transitive  (instead of being Cayley graphs). 

\medskip
  
  Take \emph{any} finitely generated infinite simple group $\Gamma$ and \emph{any} finite set $S$ of generators (if we want to give a negative answer to Problem \ref{pro:5}, we only ask that $S$ contains an element of order 2). Write $S=s_1,\ldots,s_k$ and consider the graph $G$ obtained from $\Cay(\Gamma,S)$ by doing the following: for any vertex $g\in \Gamma$ and element $s_i\in S$ we add a path of length (number of edges) $3i$ between $g$ and $gs_i$, say $P_{g,s_i}=v_0,v_1,\ldots, v_{3i}$ with $v_0=g$ and $v_{3i}=gs_i$, and we add a new vertex whose unique neighbor is $v_{3i-1}$. We also keep in $G$ the (simple) edges of from $\Cay(\Gamma, S)$. Note that $G$ is quasi-transitive, and we claim that $\Aut(G)=\Gamma$. To see this, observe first that for the labelled, directed version $\hat{G}$ of $\Cay(\Gamma,S)$ where arcs are labelled with the corresponding generators (i.e., in $\hat{G}$ we add an arc labelled $s$ from $g$ to $gs$ for any $g\in \Gamma$ and $s\in S$), we have $\Aut(\hat{G})=\Gamma$. Observe now that the paths $P_{g,s_i}$ in $G$ force every automorphism of $G$ to coincide with an automorphism of $\hat{G}$ on $V(\hat{G})$.

\medskip

  A proof along the lines of that of Lemma \ref{lem:2} now shows that
  $G$ has no periodic proper vertex-coloring, and if $S$ contains an
  element of order 2, then $G$ has no periodic orientation. On the other hand, $G$ has a periodic non-trivial
  vertex-coloring (obtained by coloring all the vertices of the
  original Cayley graph $\Cay(\Gamma,S)$ with color 1, and all the
  newly added internal vertices of the paths with color 2). More
  generally, any quasi-transitive graph which is not vertex-transitive
  has a periodic non-trivial vertex-coloring.

\section{Infinitely-ended graphs and finite treewidth}\label{sec:tw}

Recall from Section \ref{sec:prel} that the number of ends of a
quasi-transitive graph is either $0,1,2$ or $\infty$. The
graphs presented in Section \ref{sec:example} providing negative answers
to Problems \ref{pro:4} and \ref{pro:5} are all quasi-isometric to finitely
generated infinite simple groups and are thus 1-ended \cite{Loss86}. 
%\ugo{More detailed arguments: 
%By \cite[Corollary 1.3]{AMO07}, any $\infty$-ended finitely generated group $\Gamma$ is \emph{SQ-universal}, i.e., every countable group can be embedded in a quotient quotient of $\Gamma$. On the one hand, if we assume that $\Gamma$ is simple and SQ-universal, then every nontrivial countable group must be a subgroup of $\Gamma$, as the only quotients of $\Gamma$ are the trivial group and $\Gamma$. On the other hand, Neumann \cite{Neumann37} constructed an uncountable family of pairwise non-isomorphic groups of rank $2$. It implies that no countable group have all groups of rank $2$ as subgroups, so in particular, infinite simple groups must be $1$-ended.  
%}
It is thus natural to investigate whether we
can also construct examples with 2 or infinitely many ends (graphs
with 0 ends are finite, and therefore not of interest for these
problems). In this section we will construct $\infty$-ended
graphs providing  negative answers
to Problems \ref{pro:4} and \ref{pro:5}, and in the next section we
will show on the other hand that the answer to both problems is
positive for 2-ended graphs. More precisely, the negative examples we
will construct in this section have bounded treewidth, while all 
locally finite quasi-transitive 2-ended graphs have bounded
pathwidth, so this shows that graphs of bounded pathwidth and bounded treewidth behave very
differently in the context of our problems.

\subsection{Graphs of bounded treewidth}
\label{sec: example}
We now describe a simple construction of a quasi-transitive locally
finite graph of treewidth 2 that does not admit any periodic
orientation. This provides a negative answer to Problem \ref{pro:5}, and
thus also to Problems \ref{pro:4} and \ref{pro:pcplanar}, as the graph is planar. This example is a particular case
of a more general family of examples that we will describe in the next
subsection. We start by giving a self-contained proof that the 
graph does not admit a periodic orientation, and then we show more
generally that this holds for a wider family of examples.

\medskip

Consider an orientation $\protect\overrightarrow{T}$ of the infinite 3-regular
tree $T$ where each vertex has in-degree 2 and out-degree 1. Next, for
every arc $(u,v)$ in $\protect\overrightarrow{T}$, replace $(u,v)$ by a (non-oriented)
path $u-r_{uv}-s_{uv}-v$, and add a new vertex $t_{uv}$ adjacent to $s_{uv}$ only. Finally, for each original vertex $u$ of
$\protect\overrightarrow{T}$, replace $u$ by a triangle  $\triangle_u=u_1u_2u_3$, where each $u_i$ is adjacent to a unique neighbor of~$u$ and where $u_1$ is the closest vertex from the unique
out-neighbor of $u$ in $\protect\overrightarrow{T}$. The resulting graph
is denoted by $G$ (see Figure \ref{fig: tree} for an illustration).

\begin{figure}[htb]
  \centering
  \includegraphics[scale=1]{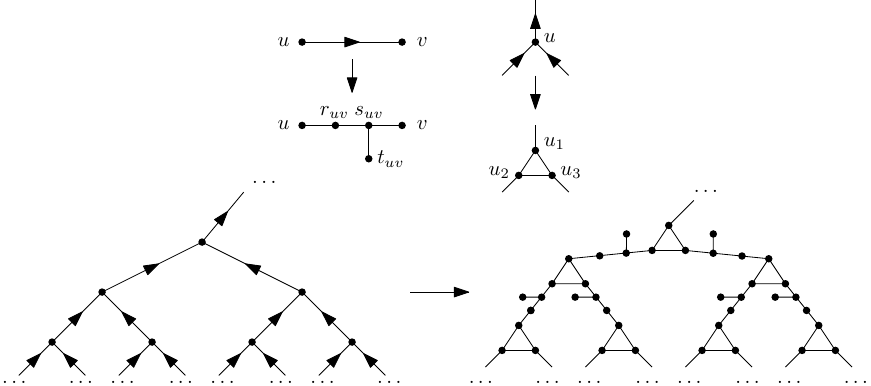}
  \caption{The orientation  $\protect\overrightarrow{T}$ of the infinite 3-regular
    tree  $T$ (left), and the resulting graph $G$ (right). The local
    modifications performed to construct $G$ from $\protect\overrightarrow{T}$
    are illustrated at the top.}
  \label{fig: tree} 
\end{figure}

It is easily seen that $G$ has treewidth at most 2 (and is planar).
The fact that $G$ is quasi-transitive easily follows from the
fact that $\protect\overrightarrow{T}$ is vertex-transitive, and that there is
a natural bijection between the 
automorphisms of $\protect\overrightarrow{T}$ and the automorphisms of
$G$: given an automorphism $g$ of $\protect\overrightarrow{T}$,
the corresponding automorphism of $G$ maps $u_1$ to $g(u)_1$ and $\{u_2,u_3\}$ to $\{g(u)_2,g(u)_3\}$ for any vertex $u\in V(\protect\overrightarrow{T})$, and $r_{uv},s_{uv}, t_{uv}$ to $r_{g(u)g(v)},s_{g(u)g(v)},
t_{g(u)g(v)}$ for any arc $(u,v)$ of  $\protect\overrightarrow{T}$. In the
other direction, observe that any automorphism $h$ of $G$ must map each
triangle $\triangle_u$ to some triangle $\triangle_{g(u)}$, and the
resulting map $g$  is an automorphism of $\protect\overrightarrow{T}$.

\begin{remark}\label{rem:halin}
  Every automorphism of  $\protect\overrightarrow{T}$ 
translates some 2-way infinite directed path or stabilizes some
vertex (this is a classical result of Tits \cite[Proposition~3.2]{Tits70}, but follows also easily from a result of Halin \cite[Theorem~7]{Halin73}). In the latter case, if
the automorphism is distinct from the identity, then there is a vertex
$v$ such that the two in-neighbors of $v$ are exchanged by the
automorphism. 
\end{remark}

This remark implies the following.

\begin{claim}
\label{clm: invert}
 For every subgroup $\Gamma$ of $\Aut(G)$ acting quasi-transitively on $V(G)$, there
 exists some automorphism $g\in \Gamma$ and some $u\in V(\protect\overrightarrow{T})$ such that $g(u_2)=u_3$ and $g(u_3)=u_2$.
\end{claim}

\begin{proofofclaim}
  Recall that there is a natural bijection  between $\Aut(G)$ and
  $\Aut(\protect\overrightarrow{T})$. If some automorphism $g \in \Gamma$
  distinct from the identity has
  the property that the corresponding automorphism of
  $\protect\overrightarrow{T}$ stabilizes a vertex of $\protect\overrightarrow{T}$, then
  the result follows from Remark \ref{rem:halin}. So we only need to
  prove the existence of such an automorphism $g \in \Gamma$.

  We say that $v$ is an \emph{ancestor} of $u$ in $\protect\overrightarrow{T}$ if
  there is a directed path from $u$ to $v$ in
  $\protect\overrightarrow{T}$. The \emph{minimum common ancestor} of two vertices $u$
  and $v$ refers to the common ancestor of $u$ and $v$ which is minimum in the
  ancestor relation (viewed as a partial order on $V(\protect\overrightarrow{T})$). We define a relation $\sim$ on $V(\protect\overrightarrow{T})$ as
  follows. Given two vertices $u,v\in V(\protect\overrightarrow{T})$, $u\sim
  v$ if and only if $d_T(u,w)=d_T(v,w)$, where $w$ denotes the minimum
  common ancestor of $u$ and $v$ and the function $d_T$ denotes the
  distance in $T$ (or equivalently  $\protect\overrightarrow{T}$). Note that $\sim$ is an
  equivalence relation, and each equivalence class is infinite. In
  particular, since $\Gamma$ acts quasi-transitively on $V(G)$, there
  exist an automorphism $g\in \Gamma$ such that the corresponding
  automorphism of $\protect\overrightarrow{T}$ maps some vertex $u$ to some
  different vertex $v\sim u$. But then such a non-trivial automorphism must
  stabilize all common ancestors of $u$ and $v$ in
  $\protect\overrightarrow{T}$, as desired.
\end{proofofclaim}

Now, assume for the sake of contradiction that $G$ has a periodic
orientation. Then, by definition, the subgroup of
orientation-preserving automorphisms of $G$ acts quasi-transitively on
$V(G)$. By Claim \ref{clm: invert}, some orientation-preserving
automorphism $g$ of $G$ must exchange two vertices $u_2$ and $u_3$, for some $u\in
V(\protect\overrightarrow{T})$. Assume by symmetry that $(u_2,u_3)$ is an
arc in the periodic orientation of $G$. Then $g$ maps the arc $(u_2,u_3)$ to
$(u_3,u_2)$, which contradicts the property that $g$ is
orientation-preserving (cf. \cref{lem:edge-inversion}). This contradiction shows  that $G$
does not admit a periodic orientation (and thus does not admit a periodic proper
vertex-coloring either).

\subsection*{Edge-colorings}

We now describe a similar construction of a quasi-transitive locally finite tree that does not admit any periodic proper edge-coloring. Consider the tree $T'$ obtained from the orientation $\vec{T}$ of the infinite $3$-regular tree we described above, after replacing every arc $(u,v)$ by a (non-oriented) path $u-r_{uv}-s_{uv}-v$ and after adding a vertex $t_{uv}$ adjacent to $s_{uv}$. Equivalently, $T'$ is obtained from the graph $G$ we constructed above after contracting every triangle $\Delta_{u}$ into a single vertex $u$ (see Figure \ref{fig: tree2}).
Then the exact same arguments used in the proof of Claim \ref{clm: invert} apply to show that for every subgroup $\Gamma$ of $\Aut(T')$ acting quasi-transitively on $T'$, there exists some non-trivial element $g\in \Gamma\setminus \sg{1_{\Gamma}}$ that fixes a vertex $w\in V(T)$ while exchanging two of its neighbors. In particular, $g$ maps some edge $e$ to some edge $f$ incident to $e$, which contradicts the fact that $g$ is color-preserving (since $e$ and $f$ have distinct colors in any proper edge-coloring). This contradiction shows that $T'$ cannot admit a periodic proper edge-coloring. 

\begin{figure}[htb]
  \centering
  \includegraphics[scale=1]{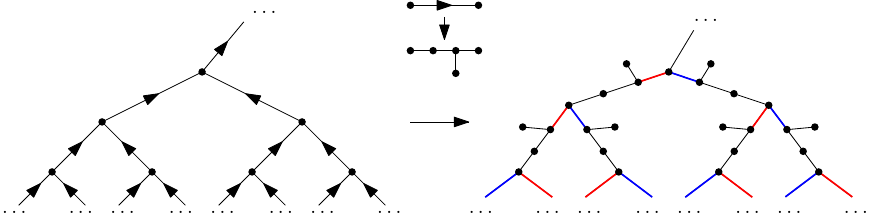}
  \caption{The construction of a quasi-transitive tree $T'$ with no periodic proper edge-coloring (right), starting from the orientation $\protect\overrightarrow{T}$ of the infinite 3-regular tree (left). In every periodic edge-coloring of $T'$, some color-preserving automorphism of $T'$ has to exchange a pair of incident edges (depicted in red and blue on the figure).}
  \label{fig: tree2} 
\end{figure}

\subsection{Obstructions to periodic orientations and colorings}
\label{sec: obstructions}

We say that a
tree-decomposition $(T,\mathcal X)$ of a graph $G$ is \emph{canonical}, if $\Aut(G)$ induces a group action by automorphisms on $T$ such that for every $g
\in \Aut(G)$ and $t\in V(T)$, $g(X_t)= X_{g\cdot t}$. Here $t\mapsto  g\cdot t$ denotes the action of $g\in\Aut(G)$ on $t\in V(T)$.
In particular, every automorphism of $G$ sends bags of $(T,\mathcal X)$ to bags and
adhesion sets to adhesion sets.

\medskip

Recall that the action of a group $\Gamma$ on a set $X$  is \emph{free} if the only
element of $\Gamma$ stabilizing an element of $X$ is the identity
element $1_\Gamma$. We will need the following result, whose proof is based on the same arguments we used to prove Claim \ref{clm: invert}.

\begin{lemma}\label{lem:mat}
Let $\Gamma$ be a group acting quasi-transitively on a $\infty$-ended tree $T$ and stabilizing an
end of $T$. Then the action of $\Gamma$ on $T$ cannot be free.
\end{lemma}

\begin{proof}
Let $\omega$ denote an end of $T$ fixed by the action of $\Gamma$ on $T$, and consider the orientation $\vec{T}$ of $T$ obtained after orienting every edge $e=uv$ toward $\omega$, i.e., if $T_u, T_v$ denote the two components of $T-e$ containing respectively $u$ and $v$,
then we add the arc $(u,v)$ in $\vec T$ if $\omega$ lives in $T_v$. Now every vertex of $\vec{T}$ has out-degree exactly $1$, and $\Gamma$ induces a group action on $\vec{T}$.
%By a result of Halin \cite{Halin73}, automorphisms of a tree fall into
%two classes: either they fix a vertex or an edge (they are then said
%to be \emph{elliptic}) or they act on a double ray as a translation
%(they are then said to be \emph{hyperbolic}).

%Note that for any double ray $p$ of $T$, if $\vec{p}$ denotes the orientation of $p$ in $\vec{T}$,
%there is at most one vertex of $p$ with in-degree $2$ in $\vec{p}$.
%In particular, it implies that every hyperbolic element $g\in \Gamma$ can only translate double rays $p$ with no such vertex, i.e., double rays pointing toward $\omega$.   
%Moreover, note that if an element $g\in \Gamma$ fixes an edge, then it must also fix its two endpoints. Hence we only need to prove that $\Gamma$ contains a non-trivial elliptic element in order to obtain the desired result.

The remainder of the proof  is identical to the one in the proof of Claim \ref{clm: invert}: we define similarly an equivalence relation $\sim$ on $\vec{T}$ by letting $u\sim v$ if and only if $u$ and $v$ lie at the same distance from their minimum common ancestor. As $T$ is quasi-transitive and has infinitely many ends, the equivalence classes of $\sim$ must be infinite. In particular, as $\Gamma$ acts quasi-transitively on $T$, there exist $u,v\in V(T)$ with $u\neq v$, $u\sim v$ and some $g\in \Gamma$ such that $g(u)=v$. Then $g$ is non-trivial, and stabilizes every common ancestor of $u$ and $v$, so it follows that the action of $\Gamma$ on $T$ is not free.
\end{proof}

The results of this subsection deal with the situation of locally finite infinitely ended graphs with an automorphism group fixing an end and acting transitively on the graph. It is known \cite{M92} that these graphs are quasi-isometric to trees and hence have bounded treewidth. 
It follows that the same is true if the action fixes an end but only acts quasi-transitively on the graph.%\ugo{Don't you also need the number of ends to be infinite?}\textcolor{green}{M: Yes, sorry. I added it. Also, I just realized, both references ask for a transitive action. Please say if you need an argument for the last sentence.}\ugo{I am fine with just mentionning it the way you did.}

The example presented in the previous subsection is a particular case of a more general
family of graphs whose structure is described by the following
lemma. 

\begin{lemma}\label{lem:nsc}
	Let $G$ be a locally finite quasi-transitive graph and let $(T, \X)$ be a canonical tree-decomposition of $G$ such that:
	\begin{itemize} 
		\item for  every automorphism $g$ of $G$ that stabilizes the bag of
                  some vertex $t\in V(T)$ but does not stabilize the
                  bag of some
                  neighbor of $t$ in $T$, $g$ exchanges
                  two adjacent vertices in $G$,
                  and
                  \item there is an end of $G$ which is stabilized by
                    all automorphisms of $G$.
                  \end{itemize}
                  Then $G$ does not have a periodic orientation (and
                  thus does not have a periodic proper vertex-coloring either).
 %Suppose that $G$ has a symmetric coloring. Then, the group $\Gamma \subseteq \Aut(G)$ that respects the coloring acts freely on $T$. 
\end{lemma}
\begin{proof}
Assume for the sake of contradiction that $G$ has a periodic
orientation, and let $\Gamma$ be the group of orientation-preserving
automorphisms of $G$, acting quasi-transitively on $G$, and thus also on
$T$. Note that $\Gamma$ stabilizes the end of $T$ corresponding to the
end of $G$ stabilized by $\Aut(G)$. By Lemma \ref{lem:mat}, the action
of $\Gamma$ on $T$ cannot be free, so there is an automorphism $g\in
\Gamma$ that stabilizes some bag $X_t$ of  $(T, \X)$, but does not
stabilize some bag $X_{t'}$ where $t'$ is a neighbor of $t$ in $T$. It
follows that $g$ inverts an edge $uv$ in $G$, which
contradicts the fact that $\Gamma$ is orientation-preserving.
\end{proof}

The following slightly different class of graphs
also provides examples of locally finite quasi-transitive graphs of bounded
treewidth with no periodic proper vertex-coloring.

\begin{lemma}
	Let $G$ be a locally finite $\infty$-ended quasi-transitive graph, and let $(T, \X)$ be a canonical tree-decomposition of $G$ such that:
	\begin{itemize}
		\item  every bag is a finite clique, and 
		\item for every adhesion set $X$ there is a unique edge
                  $tt'\in E(T)$ such that $X$ is the intersection of
                  the bags of $t$ and $t'$,
                  and
                  \item there is an end of $G$ which is stabilized by
                    all automorphisms of $G$.
                  \end{itemize}
                  Then $G$ does not have a periodic proper vertex-coloring.
\end{lemma}
\begin{proof}
 Assume for the sake of contradiction that $G$ has a periodic proper
 vertex-coloring $c$, and let $\Gamma$ be the subgroup of
 color-preserving automorphisms of $G$ (acting quasi-transitively on
 $G$, and thus also on $T$).
For $g \in \Gamma$, we denote by $g^T$ the action of $g$ on
$T$. Suppose for a contradiction that $\Gamma$ does not act freely on
$T$. Then, there exists an edge $tt' \in E(T)$ and $g_1, g_2 \in
\Gamma$ such that $g_1^T(t) = g_2^T(t)$ but $g_1^T(t') \neq
g_2^T(t')$. Since the bag $X_t$ of $t$ in $(T, \X)$  is a clique, every vertex of
$X_t$ is colored with a distinct color, and since $\Gamma$ respects
the coloring, it follows that $g_1(v) = g_2(v)$ for all $v \in
X_t$. In particular, $g_1$ and $g_2$ agree on $X_t \cap X_{t'}$, the
adhesion set of $(T, \X)$ corresponding to the edge $tt'$ of $T$. But
since the adhesion sets of $X_t$ are unique, it follows that
$g_1^T(t') = g_2^T(t')$, a contradiction. This shows that $\Gamma$
acts freely on $T$, a contradiction with Lemma \ref{lem:mat}.
\end{proof}

\subsection{Limits of color reduction}

 Lemma \ref{lem:cr} shows that Problem \ref{pro:4} is equivalent to the stronger version where we ask for a periodic proper $(\Delta(G)+1)$-vertex-coloring. This raises the following natural question.

  \begin{problem}\label{pro:pc}
      Is it true that every locally finite graph $G$ with a periodic proper vertex-coloring has a periodic proper vertex-coloring with $\chi(G)$ colors?
  \end{problem}

  We will see in the next section that the answer to this question is positive for 2-ended graphs, but in the remainder of this section we show that the answer is negative for $\infty$-ended graphs.

    \begin{figure}[htb]
  \centering
  \includegraphics[scale=1.15]{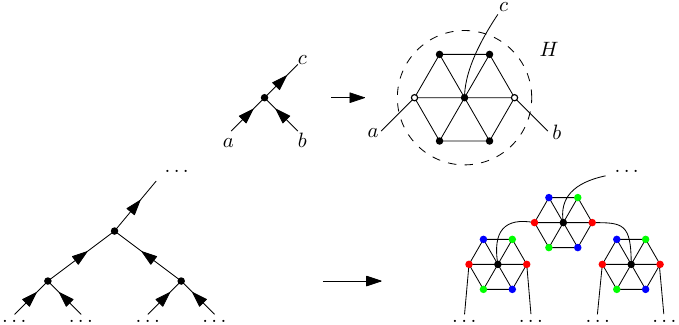}
  \caption{The construction of a quasi-transitive graph $G$ with no periodic proper vertex-coloring with $\chi(G)=3$ colors but with a periodic proper vertex-coloring with $\chi(G)+1=4$ colors (right), starting from an orientation $\protect\overrightarrow{T}$ of the infinite 3-regular tree (left). }
  \label{fig: tree3} 
\end{figure}

 % \begin{figure}[htb]
  %\centering
  %\includegraphics[scale=1.15]{tw3}
 % \caption{The construction of a quasi-transitive graph $G$ with no periodic proper vertex-coloring with $\chi(G)=3$ colors but with a periodic proper vertex-coloring with $\chi(G)+1=4$ colors (right), starting from an orientation $\protect\overrightarrow{T}$ of the infinite 4-regular tree (left). }
 % \label{fig: tree3} 
%\end{figure}

Consider the unique orientation $\protect\overrightarrow{T}$ of the infinite 3-regular tree where every vertex has out-degree exactly one (see Figure \ref{fig: tree3}, left), and replace every vertex $v\in V(\protect\overrightarrow{T})$ by a copy $H_v$ of the finite graph $H$ depicted in Figure \ref{fig: tree3}, top right, connecting incident edges as illustrated there. The resulting graph $G$ is depicted in Figure \ref{fig: tree3}, right. This graph is locally finite, $\infty$-ended, quasi-transitive, and satisfies $\chi(G)=3$. Figure \ref{fig: tree3} (right) illustrates the fact that $G$ has a periodic proper vertex-coloring with $4=\chi(G)+1$ colors. On the other hand, we claim that $G$ does not have such a periodic coloring with $3=\chi(G)$ colors. To see this, we first observe that automorphisms of $G$ map copies of $H$ to copies of $H$, so any automorphism of $G$ induces an automorphism of $\protect\overrightarrow{T}$. Moreover, every automorphism of $\protect\overrightarrow{T}$ must stabilize the end $\omega$ of $\protect\overrightarrow{T}$ containing  all infinite 1-way directed paths. 

%Consider the unique orientation $\protect\overrightarrow{T}$ of the infinite 4-regular tree where every vertex has out-degree exactly one (see Figure \ref{fig: tree3}, left), and replace every vertex $v\in V(\protect\overrightarrow{T})$ by a copy $H_v$ of the finite graph $H$ depicted in Figure \ref{fig: tree3}, top right, connecting incident edges as illustrated there. The resulting graph $G$ is depicted in Figure \ref{fig: tree3}, right. This graph is locally finite, $\infty$-ended, quasi-transitive, and satisfies $\chi(G)=3$. Figure \ref{fig: tree3} (right) illustrates the fact that $G$ has a periodic proper vertex-coloring with $4=\chi(G)+1$ colors. On the other hand, we claim that $G$ does not have such a periodic coloring with $3=\chi(G)$ colors. To see this, we first observe that automorphisms of $G$ map copies of $H$ to copies of $H$, so any automorphism of $G$ induces an automorphism of $\protect\overrightarrow{T}$. Moreover, every automorphism of $\protect\overrightarrow{T}$ must stabilize the end $\omega$ of $\protect\overrightarrow{T}$ containing  all infinite 1-way directed paths. 

Assume for the sake of contradiction that $G$ has a periodic proper coloring $c$ with 3 colors, and let $\Gamma$ be the subgroup of $\Aut(G)$ of automorphisms preserving the coloring $c$. By Lemma \ref{lem:mat}, the action
of $\Gamma$ on $\protect\overrightarrow{T}$ cannot be free, so there is an automorphism $g\in
\Gamma\setminus \sg{1_{\Gamma}}$ whose induced action on $\overrightarrow T$ stabilizes some vertex $v\in \overrightarrow T$, but does not
stabilize some in-neighbor $u$ of $v$ in $\overrightarrow T$. In particular, $g$ stabilizes $H_v$, and maps $H_u$ to some disjoint copy $H_{u'}$ of $H$, for some $u'\neq u$. This automorphism must exchange the two vertices of $H_v$ connected to the copies of $H$ associated to the in-neighbors of $v$ in $\protect\overrightarrow{T}$ (these vertices are depicted in white in Figure \ref{fig: tree3}, top right). However it can be checked easily that in any proper vertex-coloring of $H$ with $\chi(H)=3$ colors, these 2 vertices have distinct colors, which contradicts the fact that $g$ is color-preserving. This shows that $G$ does not have a periodic proper vertex-coloring with $\chi(G)$ colors, while it has a periodic proper vertex-coloring with $\chi(G)+1$ colors.

\begin{remark}
    Consider the graph $G$ described above, and replace in $G$ every vertex by a clique of size $k$, and every edge be a complete bipartite graph $K_{k,k}$. Then the resulting graph has chromatic number $3k$, and a proof similar to that above shows that every periodic proper vertex-coloring requires at least $4k$ colors. 
\end{remark}

\begin{remark}
     We have proved in Lemma \ref{lem:cr} that any graph $G$ with a periodic proper vertex-coloring has a periodic proper vertex-coloring with $\Delta(G)+1$ colors. The example constructed above has maximum degree $\Delta(G)=4$, and no periodic proper vertex-coloring with $3=\Delta(G)-1$ colors, which shows that Lemma \ref{lem:cr} is close to best possible.
\end{remark}

Our construction in this section gives a negative answer to Problem \ref{pro:pc} as it stands, but the following weaker problem remains open (even for $f(x)=\tfrac43x+1$).

\begin{problem}\label{pro:pc2}
      Is there a function $f$ such that every locally finite graph $G$ with a periodic proper vertex-coloring has a periodic proper vertex-coloring with $f(\chi(G))$ colors?
  \end{problem}
  
  \section{2-ended graphs and bounded pathwidth}\label{sec:pw}

 In previous sections we have seen  1-ended and $\infty$-ended examples of
    quasi-transitive locally finite graphs with no periodic
    orientation (and thus no periodic proper vertex-coloring). In
    this section, we
    show that every 2-ended quasi-transitive locally finite graph has
    a periodic vertex-coloring (and thus a periodic orientation). We start with the case of Cayley
    graphs of finitely generated groups, for which the result directly
    follows from classical results on subshifts of finite type in
    symbolic dynamics. We then show that the same ideas can be applied
    more generally to any 2-ended quasi-transitive locally finite graph.
  
\subsection{2-ended groups and subshifts of finite type}

  Consider a finitely generated group $\Gamma$. Let $A$ be a finite set of colors and consider any coloring $\sigma: \Gamma\to A$. For every $g\in \Gamma$, we define the coloring $g\cdot \sigma: \Gamma \to A$ by setting for each $x\in \Gamma$, $g\cdot \sigma (x):=\sigma(g^{-1}x)$. Note that this defines a group action of $\Gamma$ on the set $A^{\Gamma}$ of $A$-colorings of $\Gamma$. 

  %In this section (and only in this section) we will view $G$ as an edge-labelled graph, where for every $g\in \Gamma$ and $s\in S$ the edge between $g$ and $gs$ in $G$ is oriented from $g$ to $gs$ and labelled with $s$.  
   Let $F$ be a finite subset of $\Gamma$, and let $\alpha: F\to A$ be a coloring of $F$. A coloring $\sigma: \Gamma \to A$ of the elements of $\Gamma$ is said to \emph{avoid the pattern} $\alpha$ (we also say that \emph{the pattern $\alpha$ is forbidden in} $\sigma$) if for any $g\in \Gamma$, the restriction of $g\cdot \sigma$ to $F$ is distinct from $\alpha$. In other words, this means that for all sets $S$ in the $\Gamma$-orbit of $F$, the corresponding coloring $\alpha$ of $S$ is avoided in $\sigma$.
   
   A \emph{subshift of finite type in $\Gamma$} is the set of all colorings of $\Gamma$ avoiding a given finite set of patterns. We say that a coloring $\sigma: \Gamma \to A$ is \emph{strongly periodic} if $\Stab_{\Gamma}(\sigma)$ has finite index in $\Gamma$, or equivalently if the orbit $\Gamma \cdot \sigma$ is finite.
   We say that $\Gamma$ is \emph{strongly periodic} if any non-empty subshift of finite type in $G$ contains a strongly periodic coloring. 

\begin{remark}
 \label{rem: QT-SFT}
 Let $\sigma: \Gamma \to A$, and fix a finite generating set $S$ of
 $\Gamma$. Then $\sigma$ is strongly periodic if and only if $\Gamma$
 induces a quasi-transitive group action on the colored graph
 $(\Cay(\Gamma,S), \sigma)$. This follows from the fact that the
 elements of $\Gamma$ that induce automorphisms of the colored graph
 $(\Cay(\Gamma,S), \sigma)$ are exactly the elements from
 $\Stab_{\Gamma}(\sigma)$. Thus the $\Gamma$-orbits of the vertices of
 the colored graph $(\Cay(\Gamma,S), \sigma)$ correspond to the
 different right-cosets $\Stab_{\Gamma}(\sigma)g$. (This is just a
 rephrasing of the arguments in the proofs of Lemmas \ref{lem:sptop} and \ref{lem:ptosp}
 in Section \ref{sec:prel}.)
\end{remark}

It was proved in \cite{CP15} that every group which is virtually $\mathbb{Z}$  is strongly periodic (see also \cite{Coh17}). This directly implies the following, which gives a positive answer to Problem \ref{pro:pc} in the case of Cayley graphs of bounded pathwidth.

   \begin{theorem}\label{thm:2endedgroup}
For every finitely generated 2-ended group $\Gamma$ and every finite
generating set $S$, the graph $G=\Cay(\Gamma,S)$ has a periodic
proper  vertex-coloring with $\chi(G)$ colors.
   \end{theorem}

   \begin{proof}
       Let $c$ be an optimal proper vertex-coloring of $G$, with
       colors from a finite set $A$ with $|A|=\chi(G)$, and
       consider the subshift of finite type $\mathcal{X}$ in $\Gamma$
       where for each edge $uv$ in $G$, the vertices $u$ and $v$ are
       required to have different colors. This can clearly be encoded
       by a finite set of forbidden patterns in $\Gamma$ (all
       colorings of pairs $(1_\Gamma,s)$ for $s\in S$ where $1_\Gamma$
       and $s$ have the same color). The coloring $c$ witnesses that
       the subshift of finite type $\mathcal{X}$ defined above is
       non-empty. Since $\Gamma$ is a finitely generated 2-ended
       group, it is virtually $\mathbb{Z}$ and thus strongly periodic
       \cite{CP15}. It follows that $G$ has a proper vertex-coloring
       $c'$ with colors from $A$ which has a finite orbit under the
       action of $\Gamma$. By Remark \ref{rem: QT-SFT}, this implies in
       particular that $c'$ is periodic, as desired.
%       The remainder of the proof is similar to the proof of Lemma \ref{lem:1}. The action of $\Gamma$ on the finite orbit $\mathcal{C}$ of $c'$  induces a permutation of $\mathcal{C}$, and this action is a group homomorphism $f$ from $\Gamma$ to $\mathrm{Sym}(\mathcal{C})$ (which is finite). The kernel $\ker f$ of $f$ consists of all elements $g\in \Gamma$ which stabilize every coloring of $\mathcal{C}$. By the first isomorphism theorem,  the quotient group $\Gamma/\ker f$ is isomorphic to the image of $f$ (which is a finite subgroup of $\mathrm{Sym}(\mathcal{C})$). Hence, $\ker f$ is a subgroup of finite index in $\Gamma$. The action of any element of $\ker f$ leaves $c'$ invariant in $G$, so the subgroup of automorphisms of $G$ preserving $c'$ contains $\ker f$, and thus also has finite index. The finitely many cosets correspond to the orbits of $G$ under the action of automorphisms of $G$ preserving $c'$, and thus $c'$ is a quasi-transitive proper coloring of $G$ with  $|A|=\chi(G)$ colors, as desired.
     \end{proof}

    We note here that the approach using subshifts of finite type is
    inherently limited: it was proved in \cite{Pia08} that non-abelian
    free groups  are not strongly periodic, and it was even
    conjectured in \cite{CP15} that a group is strongly periodic if
    and only if it is virtually cyclic (that is, finite or virtually $\mathbb{Z}$).

    \subsection{2-ended graphs and bounded pathwidth}

We now extend Theorem \ref{thm:2endedgroup} to all 2-ended locally finite quasi-transitive graphs (or equivalently, to all locally finite quasi-transitive graphs of bounded pathwidth).
Recall that a finitely generated group is $2$-ended if and only if it is virtually $\mathbb{Z}$ \cite{Hopf44}.

\medskip

A \emph{separation} in a graph $G$ is a triple $\Sep$ where
$Y,S,Z$ are pairwise disjoint subsets of $V(G)$ with $V(G)=Y\cup
S\cup Z$ and no edge of $G$ has  an endpoint in $Y$ and the
other in $Z$. The \emph{order} of $\Sep$ is $|S|$. A separation $\Sep$
is said to \emph{separate} two ends $\omega_1,\omega_2$ of $G$ if all
but finitely many vertices of each ray of $\omega_1$ lie in $Y$, and
similarly all
but finitely many vertices of each ray of $\omega_2$ lie in $Z$. 

The following is a folklore result (see \cite[Theorem 1.1]{Jun81} for a weaker version for vertex-transitive graphs, with essentially the same proof).

\begin{lemma}
 \label{lem: 2-ends}
 Let $G$ be a connected locally finite quasi-transitive graph with two
 ends. Then $G$ has a separation $\Sep$ of finite order separating the
 two ends of $G$ and there is an element $g\in \Aut(G)$ of infinite order such that $g(S\cup Z)\subseteq Z$.
\end{lemma}

\begin{proof}
 Let $k\in \mathbb N\setminus \sg{0}$ be the minimum order of a separation $\Sep$
 separating the two ends of $G$, i.e. such that both $G[Y]$ and $G[Z]$
 contain an infinite component of $G-S$, and let $\Sep$ be a separation of order
 $k$ separating the two ends of $G$ such that $G[Z]$ is connected. %Note that $k$ is finite since $G$ has bounded pathwidth. \textcolor{purple}{Tara: why is $k$ finite?}\ugo{It should follow from the fact that $G$ has $2$ ends, no?} \textcolor{purple}{Tara: I believe so, I added a sentence just making it clear}\ugo{I still think we do not really need to mention pathwidth, and that the definition of ends is enough: by definition, if two ends $\omega, \omega'$ are distinct, then there is a separation $\Sep$ of finite order and two infinite components $C_X\subseteq X,C_Y \subseteq Y$, such that the rays of $\omega$ all live in $C_X$ while the rays of $\omega'$ all live in $C_Y$. Did I miss something?} \textcolor{purple}{Tara: Ah yes I see it now, thanks!}
 Let $g_1\in \Aut(G)$ be such that $g_1(S)\subseteq Z$ (as $Z$ is
 infinite and quasi-transitive, such an
 element $g_1$ exists), and set $(Y_1,S_1,Z_1):=(g_1(Y), g_1(S), g_1( Z))$. If $Z_1\subseteq Z$, we set $g:=g_1$, and claim that $g$ must
 have infinite order, as an easy induction shows that then
 $g^{i+1}(S\cup Z)\subseteq g^i( Z)\subsetneq Z$ for all
 $i\geq 0$. If $Z_1$ is not included in $Z$, then as $G[Z]$ is
 connected (and thus $G[Z_1]$ also is) we must have $Y\subseteq Z_1$ and
 $Y_1\subseteq Z$. Let $g_2\in \Aut(G)$ be such that $g_2( S)\subseteq
 Y_1 \subseteq Z$, and set $(Y_2,S_2,Z_2):=(g_2(Y), g_2( S),
 g_2(Z))$. Again, if $Z_2\subseteq Y_1$, then $Z_2\subseteq Z$ and
 we conclude as before choosing $g:=g_2$; thus we can assume that we do not have $Z_2\subseteq Y_1$. As $G[Z_2]$ is connected we have $S_1\subseteq Z_2$ and thus $S_1\cup Z_1\subseteq Z_2$. We now conclude by choosing $g:=g_1 g_2^{-1}$, which satisfies $g (S_2\cup Z_2)=S_1\cup Z_1\subseteq Z_2$. Again, the fact that $g$ has infinite order immediately follows.
\end{proof}

% \begin{theorem}
%     Let $P$ be a double ray. Every group $\Gamma \subseteq Aut(P)$ that acts quasitransitively on $P$ contains a translation. %and every $d \geq 1$, there is a subgroup $\Gamma' \subseteq \Gamma$ such that $\Gamma'$ acts freely and quasitransitively on $P$, and $d' \geq d$ such that every orbit of $\Gamma'$ is a distance-$d'$ subset of $V(P)$. 
% \end{theorem}
% \begin{proof}
%     First, we claim that $\Gamma$ contains a translation. Fix vertices $u, v, w \in V(P)$ such that $u, v, w$ are in the same orbit of $\Gamma$ and $v$ is on the path from $u$ to $w$ in $P$. Let $\gamma, \gamma' \in \Gamma$ such that $\gamma(u) = v$ and $\gamma'(u) = w$. We may assume that $\gamma$ and $\gamma'$ are not translations, since otherwise we are done. Therefore $\gamma$ and $\gamma'$ are reflections about a vertex or an edge. Now, $\gamma \circ \gamma'$ is a translation. 
   % Next, we fix a translation $\gamma \in \Gamma$ and suppose $\gamma$ translates distance $d_\gamma$. Let $n$ be minimum such that $n \cdot d_\gamma \geq d$. Finally, let $\Gamma'$ be the subgroup generated by $\gamma^n$. Now $\Gamma'$ satisfies the conditions of the theorem.  
% \end{proof}

We now prove the main result of this section, giving a positive answer to Problem \ref{pro:pc} for connected locally finite 2-ended quasi-transitive graphs, and extending Theorem \ref{thm:2endedgroup}.

\begin{theorem}
\label{thm: 2-ends}
    Let $G$ be a connected locally finite quasi-transitive graph with
    $2$ ends. Then $G$
    has a periodic proper vertex-coloring with $\chi(G)$ colors.
\end{theorem}
\begin{proof}
We let $\Sep$ and $g\in \Aut(G)$ be given by Lemma \ref{lem: 2-ends}.
For each $i\in \mathbb Z$, let $\Sepi{i}:=(g^i (Y), g^i(S), g^i(Z))$. Then $S_{j}\cup Z_{j}\subseteq Z_i$ for all $i<j$ and $\Sepi{i}$ also separates the two ends of $G$. Let $c: V(G)\to [\chi(G)]$ be a proper vertex-coloring of $G$. By the pigeonhole principle there exist $i<j$ such that $c(g^i(x))=c(g^j(x))$ for all $x\in S$. Up to replacing $g$ by $g^{j-i}$, we may assume that $i=0$ and $j=1$, i.e., that $c(g(x))=c(x)$ for all $x\in S$.

For all $i\in \mathbb Z$, we let $V_i:=V(G)\setminus (Y_i\cup Z_{i+1})$. Then for each $i\in \mathbb Z$, $S_i\cup S_{i+1}\subseteq V_i$, $V_{i+1}=g(V_i)$, $V_i\cap V_{i+1}=S_{i+1}$, and as $G$ has two ends and bounded degree, the graph $G_i:=G[V_i]$ is finite. Moreover, note that $\sg{V_i: i\in \mathbb Z}$ covers $V(G)$.
We also observe that for every $i\leq i'\leq j$, we have $V_i\cap V_j=V(G)\setminus (Y_j\cup Z_{i+1})\subseteq V(G)\setminus (Y_{i'}\cup Z_{i'+1})=V_{i'}$. This last inclusion implies in particular that for every $i< j$ and $v\in V_i\cap V_j$, we have $v\in V_i\cap V_{i+1}\cap V_{j-1}\cap V_j=S_{i+1}\cap S_{j+1}$. Hence, by our choice of $\Sep$, we then have $c(g^{-i}(v))=c(g^{-j}(v))$.

We now define a vertex-coloring $\tc: V(G)\to [\chi(G)]$ by  
setting for each $i\in \mathbb Z$ and $v\in V_i$,
$\tc(v):=c(g^{-i}(v))$. 
In other words, the vertex-coloring $\tc$ is
obtained by repeating periodically $c|_{V_0}$ on each $G_i$. First,
note that by our previous remark and the fact that $\sg{V_i: i\in \mathbb Z}$ covers $V(G)$, $\tc$ is well-defined on $V(G)$. Moreover, by definition $g$ is color-preserving, i.e., $\tc(g(v))=\tc(v)$ for all $v\in V(G)$. 

We show that $\tc$ is a proper vertex-coloring. For this, we show that for every edge $uv\in E(G)$, there exists $i\in \mathbb Z$ such that $u,v\in V_i$. This simply follows from the observation that for each $i<j$, $S_{i'+1}$ separates $Y_i$ from $Z_j$. In particular, for every edge 
$uv\in E(G)$, if 
is such that $u,v\in V_j$, then  $g^{-i}(u)$ and $g^{-i}(v)$ are also adjacent in $G_0$ and thus
$\tc(u)=c(g^{-i}(u))\neq c(g^{-i}(v))=\tc(v)$, proving that $\tc$ is a proper vertex-coloring of $G$.

As the sets $V_i$ are finite and cover $V(G)$, and  $g( V_i)=V_{i+1}$
for each $i\in \mathbb Z$, the subgroup of $\Aut(G)$ generated by $g$ induces a quasi-transitive action on $V(G)$. We conclude that $G$ has a periodic proper $\chi(G)$-coloring.
%
% Let $(P, \chi)$ be a canonical path decomposition of $G$ of width $t$. Let $\Gamma = \Aut(G)$ be the automorphisms of $G$ acting on path $P$. Since $G$ is quasitransitive, it follows that $\Gamma$ acts quasitransitively on $P$. 
% 
%     Since $(P, \chi)$ has bounded width, there is $d \geq 1$ such that every vertex $v \in V(G)$ is in at most $d$ bags $\chi(p)$. 
% 
%     By Theorem 4.1, $\Gamma$ contains a translation $\gamma$. Suppose $\gamma$ translates distance $d_\gamma$. Let $n \geq 1$ be minimum such that $n \cdot d_\gamma \geq d$. Note that $\gamma^n$ is a translation of distance at least $d$. 
% 
%     Let $\gamma_G$ be an automorphism of $G$ such that the action of $\gamma_G$ on $P$ is $\gamma^n$. Let $\Gamma' \subseteq \Aut(G)$ be generated by $\gamma_G$.
%     
%     \begin{claim} $\Gamma'$ acts quasi-transitively on $G$ and every vertex $v \in V(G)$ appears at most once in each orbit of $\Gamma'$. 
%     \end{claim}
%     \begin{proofofclaim}
%         Since the action of $\gamma_G$ on $P$ is $\gamma^n$, it follows that the group generated by $\gamma^n$ is a subgroup of the action of $\Gamma'$ on $P$. (To finish) 
%     \end{proofofclaim}
% 
%     Let $P' \subseteq P$ be a subpath of $P$ consisting of a representative from each orbit of the action of $\Gamma'$ on $P$. 
\end{proof}

We recall that the \emph{chromatic index} of a graph  $G$, denoted by $\chi'(G)$, is the minimum number of colors
 in a proper edge-coloring of $G$. This is well defined when the
 maximum degree of the graph is finite, which is the case for locally
 finite quasi-transitive graphs. The \emph{line-graph} $L(G)$ of a
 graph $G$ is the graph with vertex set $E(G)$ in which two vertices
 are adjacent if and only if the corresponding edges of $G$ share a
 vertex. Note that $\chi'(G)=\chi(L(G))$ for any graph $G$.

 \smallskip
 
 We obtain the following consequence of Theorem \ref{thm: 2-ends} for
 edge-colorings of 2-ended locally finite quasi-transitive 
graphs.

\begin{corollary}
 \label{cor: edge-col}
 If $G$ is a connected locally finite quasi-transitive graph with
 2 ends, then there exists a periodic proper edge-coloring of $G$ with $\chi'(G)$ colors. 
\end{corollary}

\begin{proof}
 Note that the line-graph $L(G)$ is also locally finite, connected,
 and quasi-transitive (if there are $k$ orbits of $V(G)$ under the
 action of $\Aut(G)$, there are at most $k\Delta(G)$ orbits of
 $E(G)$ under the action of $\Aut(G)$). Moreover, as rays in $L(G)$
 are in correspondence with rays of $G$ and similarly finite subsets
 separating the ends of $G$ are in correspondence with finite subsets
 separating the ends of $L(G)$ (in both cases we use the fact that $G$
 has bounded degree), $L(G)$ is also 2-ended.
 We can thus conclude by applying Theorem \ref{thm: 2-ends} to $L(G)$.
\end{proof}

\begin{remark}\label{rem:flow}
    Corollary \ref{cor: edge-col} can also be proved directly by a simple modification of the proof of Theorem \ref{thm: 2-ends}, and by similar modifications the proof can be adapted  to work with other types of objects in graphs, for instance periodic perfect matchings (assuming $G$ has a perfect matchings), periodic nowhere-zero $k$-flows (assuming $G$ has a nowhere-zero $k$-flow), etc. In all these cases, if a connected locally finite 2-ended quasi-transitive graph $G$ can be decorated with an additional structure $\mathcal{S}$, then $G$ can be decorated with a periodic version of $\mathcal{S}$.
\end{remark}

  \section{Discussion and open problems}\label{sec:ccl}

  Recall our original question: is it true that  every quasi-transitive graph $G$ can be
decorated with a non-trivial additional structure such that the graph $G$ is still
quasi-transitive if we restrict ourselves to automorphisms preserving
the additional structure? We have considered two specific instances of this problem (Problems \ref{pro:4} and \ref{pro:5}) and proved that each of them has a negative answer, even for Cayley graphs and graphs of bounded treewidth. The examples providing a negative answer to Problem \ref{pro:4} even rule out the existence of a non-trivial vertex-coloring that is periodic, which  immediately rules out  the existence of a number of other natural structures in graphs (any non-trivial vertex subset for instance) that would be periodic, in the same sense as before.

\medskip

The corresponding problem for edge-colorings or edge-subsets appears to be more difficult. We say that an edge-coloring of a graph $G$ is \emph{trivial} if for each connected component $C$ of $G$, all the edges of $C$ are assigned the same color. We also say that a graph is \emph{trivial} if all its connected components contain at most one edge (note that every edge-coloring of a trivial graph is trivial). 

\begin{problem}\label{pro:6}
    Is it true that every non-trivial locally finite quasi-transitive graph has a
    periodic non-trivial  edge-coloring?
\end{problem}

Color reduction arguments also work for non-trivial edge-colorings. Consider a non-trivial edge-coloring $c$ of a graph $G$.
If a component $C$ of $G$ contains at least two edges, then $c$ uses at least two different colors from $C$, say colors $1,2$. For every $e\in E(G)$, let $c'(e)=1$ if $c(e)=1$, and $c'(e)=2$ otherwise. Note that $c'$ is non-trivial and only uses 2 colors. It follows that Problem \ref{pro:6} is equivalent to the following problem.

\begin{problem}\label{pro:7}
    Is it true that every non-trivial locally finite quasi-transitive graph has a
    periodic non-trivial edge-coloring with two colors?
\end{problem}

Note that coloring the edges of $G$ with two colors is equivalent to choosing a \emph{spanning} subgraph $H$ of $G$, by which we mean a subgraph of $G$ with vertex set $V(G)$ (that is, obtained from $G$ by only removing edges). This can be seen by considering the spanning subgraph of $G$ whose edges are all the edges of $G$ colored 1 (note that this spanning subgraph possibly contains isolated vertices).  So we can ask equivalently:

\begin{problem}\label{pro:8}
    Is it true that every non-trivial locally finite quasi-transitive graph has a
    periodic non-trivial spanning subgraph?
\end{problem}

Here, by \emph{non-trivial} subgraph $H$ of a graph $G$, we mean that at least
one connected component of $G$ contains an edge of $H$ and a non-edge
of $H$ (that is an edge of $G$ which is not in $H$). By \emph{periodic
  subgraph} we mean a subgraph $H$ of $G$ such that the subgroup of
automorphisms of $G$ stabilizing $H$ (that is, mapping edges of $H$ to
edges of $H$ and non-edges of $H$ to non-edges of $H$) acts
quasi-transitively on $G$.

\medskip

Note that these problems have a trivial (positive) answer for graphs that are not edge-transitive (by taking $H$ to be an orbit of the action of $\Aut(G)$ on the edge-set of $G$), so any possible counterexample must be edge-transitive.

\medskip

While we were not able to provide a negative answer to these problems,
the examples constructed in the previous sections severely restrict the properties of
such a spanning subgraph $H$ in general. Let us say that a class of
graphs $\mathcal{C}$ is \emph{periodizing} if any non-trivial locally finite
quasi-transitive graph $G$  has a
    periodic non-trivial spanning subgraph $H\in \mathcal{C}$. Problem \ref{pro:8} asks whether the class of all graphs is periodizing. We now give a number of sufficient conditions on periodizing graph classes.

\begin{lemma}\label{lem:3}
Every periodizing class contains a vertex-transitive (and thus regular) infinite graph of minimum degree at least 2. 
\end{lemma}

\begin{proof}
Consider the Cayley graph $G=\Cay(V,T)$ given by Lemma \ref{lem:V}, where we consider Thompson's group $V$, which is  a finitely generated infinite simple group, and where $T$ is a finite generating set of $V$ containing no element of order 2, and such that $G$ is a graphical regular representation of $V$ (that is, $\Aut(G)=V$, see Section \ref{sec:example}). Assume that $G$ has a non-trivial spanning subgraph $H$  such that the subgroup $\Gamma$ of automorphisms of $G$ stabilizing $H$ acts quasi-transitively on $G$. As in the proof of Lemma \ref{lem:2}, it directly follows from Corollary \ref{cor:2} that $\Gamma=V$, that is, all the automorphisms of $G$ stabilize $H$. Since $G$ is vertex-transitive, $H$ is also vertex-transitive and thus $d$-regular with $d\ge 1$, since $H$ is non-trivial. Assume for the sake of contradiction that $d=1$ and consider two adjacent vertices $u,v$ in $H$. Consider an automorphism $g$ of $G$ that maps $u$ to $v$, and observe that $g$ must map $v$ to $u$, since otherwise $v$ would have degree 2 in $H$. It then follows that the element $s\in T$ corresponding to the edge $uv$ in $G$ has order 2 in $V$, a contradiction. 
\end{proof}

A \emph{$k$-factor} in a graph is a spanning subgraph in which every vertex has degree $k$.
Observe that Lemma \ref{lem:3} implies that Problem \ref{pro:8} has a negative answer if we require $H$ to be a forest, and in particular a 1-factor (i.e., a perfect matching). The simplest version of Problem \ref{pro:8} we can think of is thus the following. 

\begin{problem}\label{pro:9}
    Is it true that every  4-regular vertex-transitive infinite graph has
    a periodic 2-factor?
\end{problem}

A classical result of Petersen \cite{Pet91} asserts that every finite Eulerian graph $G$ has a spanning subgraph $H$ such that for every vertex $v\in V(G)$, $d_H(v)=d_G(v)/2$, and the result can be extended to infinite locally finite graphs by compactness. Hence, every infinite 4-regular graph contains a 2-factor.

\subsection*{Planar graphs} Planar Cayley graphs are a well-studied topic (see \cite{GH19} and the reference therein), and by extension planar quasi-transitive graphs have been extensively studied (see for instance \cite{Babai97,HamannPlanar,EGL23,Mac23}). We have seen examples of planar quasi-transitive graphs with no periodic orientation in Section \ref{sec:tw} (these examples are all $\infty$-ended), and in Section \ref{sec:pw} we have seen that all 2-ended quasi-transitive graphs have a periodic proper vertex-coloring (and thus a periodic orientation). A natural question is whether all planar 1-ended quasi-transitive graphs have a periodic proper vertex-coloring. Using a theorem of Babai \cite{Babai97} providing embeddings of planar 1-ended 3-connected quasi-transitive graphs in the Euclidean plane or the hyperbolic plane, it can be proved that in the Euclidean case the graphs have a periodic proper vertex-coloring (see also the discussion following Question 4.8 in \cite{Tim11}). However, it is unclear whether the same holds in the hyperbolic case.

\begin{problem} Are there 1-ended quasi-transitive planar graphs without periodic proper vertex-coloring?
\end{problem}

\subsection*{Acknowledgements}
We thank Emmanuel Jeandel for his remark on the Osin group, and Sergey Norin and Piotr Przytycki for the interesting discussions and their involvement in the early stages of the project. We also thank Agelos Georgakopoulos for his comments on a preliminary version of the manuscript, and the reviewers for their helpful comments and suggestions.

\bibliographystyle{alpha}
\bibliography{biblio}

\end{document}